\documentclass[a4paper,12pt]{article}

\usepackage{color}
\usepackage{natbib}
\usepackage{amsmath}
\usepackage{hyperref}
\usepackage{indentfirst}
\usepackage{verbatim}
\usepackage{amssymb}
\usepackage{float}
\usepackage{caption}
\usepackage[english]{babel}
\DeclareRobustCommand\iff{\;\Longleftrightarrow\;}
\usepackage{rotating}
\newcommand{\x}{\boldsymbol{x}}
\newcommand{\s}{\boldsymbol{s}}

\newcommand{\M}{\boldsymbol{M}}

\newcommand{\HH}{\boldsymbol{H}}
\newcommand{\PPP}{\boldsymbol{P}}
\newcommand{\I}{\boldsymbol{I}}
\newcommand{\J}{\boldsymbol{J}}

\newcommand{\ZZZ}{\mathbb{Z}_{s_1}\times\cdots\times \mathbb{Z}_{s_r}}
\newcommand{\ZZZS}{(\mathbb{Z}_{s_1} \rtimes \mathbb{Z}^*_{s_1}) 
\times \cdots \times
(\mathbb{Z}_{s_r} \rtimes \mathbb{Z}^*_{s_r})}

\newcommand{\ZZZSZ}{(\{0\} \rtimes \mathbb{Z}^*_{s_1}) 
\times \cdots \times
(\{0\} \rtimes \mathbb{Z}^*_{s_r})}

\newcommand{\G}{\sum_{g \in G}g}
\newcommand{\HHHH}{\sum_{h \in H}h}
\newcommand{\aaa}{\boldsymbol{a}}
\newcommand{\bb}{\boldsymbol{b}}
\newcommand{\B}{\boldsymbol{B}}

\newcommand{\C}{\boldsymbol{C}}

\newcommand{\FF}{\mathbb{F}}

\newcommand{\y}{\boldsymbol{y}}

\newcommand{\PPsia}{\sum_{g \in G}a_gg}
\newcommand{\PPsib}{\sum_{g \in G}b_gg}

\newcommand{\A}{\boldsymbol{A}}

\usepackage{float}
\usepackage{amsfonts}
\usepackage{fancyhdr}
\usepackage[english]{babel}
\usepackage{longtable}
\usepackage{float}
\usepackage{caption}
\usepackage[utf8]{inputenc}
\usepackage[english]{babel}
 
\usepackage{amsthm}
\usepackage{mathtools}

\theoremstyle{definition}

\newtheorem{theorem}{Theorem}
\newtheorem{corollary}{Corollary}
\newtheorem{lemma}{Lemma}

\theoremstyle{definition}
\newtheorem{definition}{Definition}
\newtheorem{example}{Example}

\usepackage{enumerate} 
\makeatletter
\newsavebox\myboxA
\newsavebox\myboxB
\newlength\mylenA

\newcommand*\xoverline[2][0.75]{%
    \sbox{\myboxA}{$\m@th#2$}%
    \setbox\myboxB\null
    \ht\myboxB=\ht\myboxA%
    \dp\myboxB=\dp\myboxA%
    \wd\myboxB=#1\wd\myboxA
    \sbox\myboxB{$\m@th\overline{\copy\myboxB}$}
    \setlength\mylenA{\the\wd\myboxA}
    \addtolength\mylenA{-\the\wd\myboxB}%
    \ifdim\wd\myboxB<\wd\myboxA%
       \rlap{\hskip 0.5\mylenA\usebox\myboxB}{\usebox\myboxA}%
    \else
        \hskip -0.5\mylenA\rlap{\usebox\myboxA}{\hskip 0.5\mylenA\usebox\myboxB}%
    \fi}
\makeatother

\begin{document}

\bibliographystyle{plain}

\title{
Legendre $G$-array pairs and the theoretical unification of several $G$-array families
}

\author{
K. T. Arasu${}^{\rm a}$,
D. A.~Bulutoglu${}^{\rm b}$ and J. R. Hollon${}^{\rm c}$\\
${}^{\rm a}$Riverside Research, 2640 Hibiscus Way,\\ Beavercreek, OH 45431, USA\\
${}^{\rm b}$Air Force Institute of Technology, 
WPAFB, OH 45433, USA  \\
${}^{\rm c}$Applied Optimization Inc.,
3040 Presidential Dr. Suite 100\\
Fairborn, OH 45324,  USA}

\maketitle

\def\l{\lambda}
\def\cA{{\cal A}}
\def\cB{{\cal B}}
\def\cC{{\cal C}}
\def\cE{{\cal E}}
\def\cG{{\cal G}}
\def\cD{{\cal D}}
\def\cR{{\cal R}}
\def\cS{{\cal S}}
\def\cT{{\cal T}}
\def\cU{{\cal U}}
\def\cW{{\cal W}}
\def\cV{{\cal V}}
\def\cX{{\cal X}}

\def\g{\gamma}
\def\p{\pi}
\def\n{\eta}
\def\vg{\boldsymbol\g}
\def\vp{\boldsymbol\p}
\def\vn{\boldsymbol\n}
\begin{abstract}
We investigate how  Legendre $G$-array pairs are related to several different  perfect binary $G$-array families. In particular we study the relations between Legendre $G$-array pairs, Sidelnikov-Lempel-Cohn-Eastman 
$\mathbb{Z}_{q-1}$-arrays, Yamada-Pott $G$-array pairs,  Ding-Helleseth-Martinsen $\mathbb{Z}_{2}\times \mathbb{Z}_p^{m}$-arrays, Yamada $\mathbb{Z}_{(q-1)/2}$-arrays,  Szekeres $\mathbb{Z}^m_{p}$-array  pairs, Paley $\mathbb{Z}^m_{p}$-array  pairs, and Baumert $\mathbb{Z}^{m_1}_{p_1}\times \mathbb{Z}^{m_2}_{p_2}$-array  pairs. Our work also solves one of the two open problems posed in Ding~[J. Combin. Des. 16 (2008), 164-171]. Moreover, we provide several computer search based existence and non-existence results  regarding Legendre $\mathbb{Z}_n$-array pairs. Finally,  by using cyclotomic cosets, we provide a previously unknown Legendre $\mathbb{Z}_{57}$-array pair.
\end{abstract}

\section{Introduction}
\label{introduction}
In this section, we first survey several known infinite binary 
$G$-array families and $G$-array pairs for a finite abelian group $G$.
In Section~\ref{sec:results}, we show how these $G$-array 
families and $G$-array pairs are related to each other.
\subsection{$G$-arrays and their correlations}
\label{sec:introcorrelations}
Let $n$ be a positive integer and $G$ be an abelian  group of order $n$. Then  $\aaa=(a_{g})$ with $g \in  G$ and $a_g \in \mathbb{C}$ is called a {\em $G$-array}. 
 The {\em cross-correlation function}  of the two 
 $G$-arrays $(a_{g})$ and $(b_{g})$ is defined by:  
\begin{equation*}
	C_{\aaa,\bb}(t)=\sum_{g \in G} a_{gt} \bar{b}_{g},
\end{equation*}
where $t \in G$ and $\bar{b}_g$ is the complex conjugate of $b_g$. If $\aaa = \bb$, then $C_{\aaa,\aaa}(t):=C_{\aaa}(t)$ is called the {\em autocorrelation function}  of $\aaa$. 

 We call a $G$-array $\aaa$ a $\{0,1\}$ ($\{-1,1\}$) $G$-array if $a_g \in \{0,1\}$ ($\{-1,1\}$) $ \forall g \in G$. 
 In this paper, we consider only $\{-1,1\}$ or $\{0,1\}$ $G$-arrays. 
The linear transformation $a_g \rightarrow 2a_g - 1$ is a bijection  that maps a $\{0, 1\}$ $G$-array to a $\{-1, 1\}$ $G$-array.  Throughout, we
switch repeatedly between a $\{0, 1\}$ $G$-array and its corresponding $\{-1, 1\}$ $G$-array. The choice between $\{0, 1\}$ and  $\{-1, 1\}$ coefficients in any particular context is
dictated by applications or ease of computation. If we refer to a $\{0,1\}$ $G$-array as a $\{-1, 1\}$ $G$-array we mean the 
$\{-1, 1\}$ $G$-array obtained from the $\{0,1\}$ $G$-array by applying the bijection  $a_g \rightarrow 2a_g - 1$.

By the structure theorem, every finite abelian group $G$ is isomorphic to
$\ZZZ$ 
for some $r \in \mathbb{Z}^{\geq 1}$.
 Let  $H_i=\langle\omega_i\rangle$ and $|H_i|=s_i$ for $s_i \in \mathbb{Z}^{\geq 2}$. 
Then, the map $\Theta:\ZZZ\rightarrow H_1\times\cdots\times H_r$ such that
$ \Theta(\alpha_1,\ldots,\alpha_r)=\omega_1^{\alpha_1}\ldots\omega_r^{\alpha_r}$ is an isomorphism between
$\ZZZ$ and $H_1\times\cdots\times H_r$ for each set of fixed  $\{\omega_i\}_{i=1}^r$. Throughout the paper we fix the notation
 $\Theta$ for this isomorphism. 

For a $G$-array $(a_g)$ and an isomorphism 
$\Phi:G\rightarrow\Phi(G)$, define the $\Phi(G)$-array 
$\Phi(a_g)$ via
 $$\Phi((a_g))=(a'_{\Phi(g)}),\text{ where }a'_{\Phi(g)}=a_g.$$ 
 Clearly, both the autocorrelation and the cross-correlation functions are preserved under the map 
 $g \rightarrow \Phi(g)$ for any isomorphism $\Phi$, i.e. $C_{a,b}(t)=C_{\Phi(a),\Phi(b)}(\Phi(t))$ for any two $G$-arrays $(a_g)$ and
  $(b_g)$ where $g,t \in G$.   
 Also, whenever we are using an isomorphic copy of $G$ that has the form $H_1\times\cdots\times H_r$, we  say that 
$G$ is {\em written multiplicatively}, and if  $G$ has the form $\ZZZ$ we say that $G$ is {\em written additively}.
 Unless otherwise specified, for a multiplicatively (additively) written group
 we use $1$ ($0$) as the identity element. We also use $e$ as the identity element of a group $G$.  

Let $n=|G|$. Let  $\aaa=(a_g)$ be a $\{-1,1\}$ or $\{0,1\}$ $G$-array.
 Then the set $ D = \{ g \ |\ g \in G \text{ and } a_g=1 \} $ is called the {\em set of $1$ indices} of $\aaa$.
 Let $d_D(t)=|(Dt) \cap D| $, where $Dt$ is the set of elements of $D$ multiplied by $t$. Then  $d_D(t)$ is called the {\em difference function} of $D \subseteq G$, and for a  $\{0,1\}$ $G$-array $\aaa$ we have
 $$C_{\aaa}(t)=d_D(t).$$ Hence, the autocorrelation function measures how much a $\{0,1\}$ $G$-array differs from its translates.
 When $\aaa=(a_g)$ is a $\{-1,1\}$ $G$-array we get
\begin{equation} \label{eqn:EqnACF}
	C_{\aaa}(t) = n-4(k - d_D(t)),
\end{equation}
where $k=|D|$, see~\cite{PottBook}. 
 By equation (\ref{eqn:EqnACF}) if $\aaa=(a_g)$ is a $\{-1,1\}$ $G$-array,  then $$C_{\aaa}(t) \equiv n \, \: (\text{mod} \, 4).$$
  A $\{-1,1\}$ $G$-array $\aaa$  is called {\em perfect} if 
  for $t \neq e$
\begin{equation*} \label{perfect}
	C_{\aaa}(t)=\begin{cases}
		\, 0 \quad $ if $n\equiv0$ (mod $4$), $\\
		\, 1 \quad $ if $n\equiv1$ (mod $4$), $\\
		\pm 2 \,  \, \, \, \hspace{.02cm}$if $n\equiv2$ (mod $4$), $\\
		-1\, \, \, \, \hspace{.02cm}$otherwise. $
	\end{cases}
\end{equation*}
 A $\{-1,1\}$ $G$-array $\aaa=(a_g)$ is called {\em balanced} if 
  \begin{equation*} \label{balanced}
	\sum_{g \in G }a_g = \begin{cases}
		0 \quad \, \, \, \, \,    $ if $n\equiv0$ (mod $2$), $\\
		\pm1 \quad $ otherwise, $
	\end{cases}
\end{equation*}
and 
{\em almost balanced} if 
  \begin{equation*} \label{almostbalanced}
	\sum_{g \in G }a_g = \begin{cases}
		 \pm2\quad     $ if $n\equiv0$ (mod $2$), $\\
      \pm3 \quad $ otherwise. $
	\end{cases}
\end{equation*}
 Then, based on equation~(\ref{eqn:EqnACF}), 
   a $\{0,1\}$ $G$-array $\aaa$ with $\sum_{g \in G}a_g=k$  is defined to be {\em perfect} if for $t \neq e$
\begin{equation} \label{perfect01}
	C_{\aaa}(t) =d_D(t)=  \begin{cases}
		 k-\frac{n}{4} \quad \, \, $ if $n\equiv0$ (mod $4$), $\\
		 k-\frac{n-1}{4} \,\, $ if $n\equiv1$ (mod $4$), $\\
		k-\frac{n\pm 2}{4} \,  \, $ if $n\equiv2$ (mod $4$), $\\
		k-\frac{n+1}{4}\, \, \, \, $otherwise, $
	\end{cases}
\end{equation}
 and a $\{0,1\}$ $G$-array $\aaa=(a_g)$ is defined to be {\em  balanced} if 
  \begin{equation} \label{balanced01}
	\sum_{g \in G }a_g = \begin{cases}
		\frac{n}{2} \,\, \, \, \, \, \, \,
		\quad    \, $if $n\equiv0$ (mod $2$), $\\
		\frac{n \pm1}{2} \,\,  \quad \, $otherwise, $
	\end{cases}
\end{equation}
and {\em almost balanced} if
\begin{equation} \label{almostbalanced01}
	\sum_{g \in G }a_g = \begin{cases} 
		\frac{n \pm2}{2}\quad    \, $if $n\equiv0$ (mod $2$), $\\
		  \frac{n\pm3}{2} \quad \, $otherwise. $
	\end{cases}
\end{equation}

A $G$-array $\aaa$ is said to have {\em good matched autocorrelation properties} if $$\max_{t \in G\setminus\{e\}} |C_{\aaa}(t)|,$$
and $$\sum_{t \in G} |C_{\aaa}(t)|^2$$ are both small, where  $\max_{t \in G\setminus\{e\}} |C_{\aaa}(t)|$ is called the {\em peak correlation} and $\sum_{t \in G} |C_{\aaa}(t)|^2$ is called the {\em correlation energy}.

Let $G$ be a group of order $v$  
and $D$ be a subset of $G$ with $k$ elements. For any $\alpha \neq e$ and $\alpha \in G$ 
 if the equation 
 \begin{equation}\label{eqn:diff}
 d(d')^{-1} = \alpha
\end{equation}
has exactly $\lambda$ solution pairs $(d,d')$ with both $d$ and $d'$ in $D$, then the set $D$ is called a {\em difference 
set} in $G$ with parameters $(v,k,\lambda)$ denoted by
DS$(v,k,\lambda)$. If equation~(\ref{eqn:diff}) has $\lambda$ solutions for $t$ of the non-identity elements of $G$ and $\lambda+1$ solutions for every other non-identity element, then $D$ is called an  
{\em almost difference set} in $G$ with parameters $(v,k,\lambda,t)$  denoted by ADS$(v,k,\lambda,t)$.
If $G$ is an abelian (cyclic) group and $D$ is a difference set, then $D$ is called an {\em abelian (cyclic) 
difference set} in $G$. If $G$ is an abelian (cyclic) group and $D$ is an almost difference set, then $D$ is called an {\em abelian (cyclic) 
almost difference set}. 

Clearly, $D$ is a(n) (almost) difference set in $G$ if and only if  $\Phi(D)$ is a(n) (almost) difference set in $\Phi(G)$ for any isomorphism $\Phi:G\rightarrow \Phi(G)$.  For a survey of almost difference sets,  see~\cite{ArasuDing}.

A $\{-1,1\}$ $\ZZZ$-array with $k$ entries equal to $1$ and all nontrivial autocorrelation coefficients equal to $\theta=n-4(k-\lambda)$ is equivalent to an abelian DS$(n,k,\lambda)$, see Lemma 1.3 in~\cite{Jungnickel}.

Supplementary difference sets generalize the concept 
of difference sets~\cite{supplementary}.
\begin{definition}\label{Def:SDS}
Let $G$ be a group of order $v$. A collection $D_1,D_2,\ldots, D_{f}$ of $f$ subsets of $G$ with $|D_i|=k_i$
 is called a {\em supplementary difference set} in $G$ denoted by $f$-SDS($v;k_1,\ldots,$ $k_{f};\lambda$) 
if for each $\alpha \in G\setminus\{e\}$,  the constraint
$$ \alpha=xy^{-1},$$
 where  $x,y\in D_i$ for some  $i \in \{1,2,\ldots,f\}$, has exactly $\lambda$ solutions.
\end{definition}
Clearly, $D_1,\ldots,D_f$ is a  $f$-SDS($v;k_1,\ldots,$ $k_{f};\lambda$) in $G$ if and only if  $\Phi(D_1),\ldots,\Phi(D_f)$ is an  $f$-SDS($v;k_1,\ldots,$ $k_{f};\lambda$) in $\Phi(G)$ for any isomorphism $\Phi:G\rightarrow \Phi(G)$.
\subsection{The group ring notation}
First, we introduce the group ring notation that will be used in the proofs.
\begin{definition}
Let $G$ be a multiplicatively written finite abelian group and $R$ be a ring.  
  Then, the {\em group ring} of $G$ over $R$ is the set denoted by $R[G]$ defined as:  $$R[G]= \left\lbrace \sum_{g \in G} a_g g \ |\ a_g \in R\right\rbrace. $$ 
\end{definition}
$R[G]$ is a free $R$-module of rank $|G|$. Any group isomorphism 
$\Phi:G\rightarrow \Phi(G)$ extends linearly to a module  and group ring isomorphism between $R[G]$ and $R[\Phi(G)]$, where 
$$\Phi(\sum_{g \in G} a_g g)= \sum_{g \in G} 
a_{g} \Phi(g)=\sum_{\Phi(g) \in \Phi(G)} 
a_{\Phi(g)} \Phi(g).  $$ 
If $G$ is a multiplicatively written group, then multiplication and addition in $R[G]$ are defined in the same way as in the  ring of formal Laurent series 
$R[[x_1,\ldots x_n]]$. 
If $G$ is additively written, then there exists an isomorphism $\Phi:G\rightarrow\ZZZ$ for some $s_1,\ldots,s_r$. In this case, addition in $R[G]$ is defined 
the same way as in the case when $G$ is multiplicatively written. The multiplication  of two elements $u,v\in R[G]$ is defined as $$u*v=(\Theta\Phi)^{-1}(\Theta\Phi(u)\Theta\Phi(v)).$$
 For short hand notation, we define the {\em power} of a group ring element in the following way.
\begin{definition}\label{Wbar}
If $W=\sum_{ g \in G } a_g g$ is an element of $R[G]$ and $t$ some integer, then $$W^{(t)} = \sum_{ g \in G } a_g g^t,\quad \xoverline{W}= \sum_{ g \in G } \bar{a}_g g, \quad\text{and} \quad |W|=\sum_{g \in G}|a_g|.$$
\end{definition}
The following are two remarks concerning Definition~\ref{Wbar}.
\begin{enumerate}
\item For a group ring element $A$ in this paper we always have $\xoverline{A}=A$.

\item The element $\left( \sum_{ g \in G } a_g g \right)^{(t)}$ is not the same as the element  $\left( \sum_{ g \in G } a_g g \right)^{t}$.
\end{enumerate}
 Let $D
 \subseteq G$ with $|D|=k$
  and $A=\sum_{ g \in D }  g$. Then, $D$ is a DS$(v,k,\lambda)$  if and only if $$AA^{(-1)}=(k-\lambda)(1)+\lambda \left(\G\right) \; \in \; \mathbb{Z}[G].$$

We can think of a $G$-array  as a matrix. 
Let $\M$ be a matrix whose rows and columns are indexed by the elements in $G$. 
Define
\begin{equation*}
P=\{g\, |\, m_{1,g}=+1\},\label{eqnPandN}
\end{equation*}
and
\begin{equation*}
 N = \{ g \,| \, m_{1,g}=-1 \}.
\end{equation*}
Let the $G$-array $m_{1,g}$ be the  first row of $\M$. Then, 
the remaining rows of $\M$ can be obtained by setting 
\begin{equation*}
	m_{g,h}=
	\begin{cases}
		\phantom{-}1, \; \text{if} \; gh^{-1} \in P,\\
		-1, \; \text{if} \; gh^{-1} \in N.
	\end{cases}
\end{equation*}
A matrix  developed this way is called {\em $G$-developed} or {\em $G$-circulant}.

For a cyclic group $G$, if the rows and columns of a matrix $\M$ are indexed  by successive powers of a generator of $G$, then the $G$-developed 
matrix $\M$ is called {\em circulant}. 
Alternatively, a circulant matrix $\A=circ(\aaa)$ is determined by its first column, where each column (row) of $\A$ is a cyclic down (right) shift of the vector $\aaa$.
 An $m_1m_2 \times m_1m_2$ 
matrix $\C$ is said to be {\em block-circulant} if it is of the form
\begin{equation}\label{eqn:circ}
\C= circ(\C_0,\C_1,\ldots, \C_{m_2-1})=
\left[
\begin{array}{llll}
 \C_0    & \C_{m_2-1} & \cdots  & \C_1   \\
 \C_1    & \C_0    & \cdots  & \C_2   \\
 \C_2    & \C_1    & \cdots  & \C_3   \\
 \vdots  & \vdots        & \ddots & \vdots\\
 \C_{m_2-1}& \C_{m_2-2} & \cdots  & \C_0  
\end{array}
\right],
\end{equation}
where the $\C_j$ are $m_1 \times m_1$ matrices.
If each $\C_i$ in equation~(\ref{eqn:circ}) is itself also circulant then $\C$ is a block-circulant matrix of circulant matrices.
More generally, if the group $G$ is abelian but not cyclic then $G\cong \ZZZ$ for some $r\geq2$ and the $G$-developed matrix is  
{\em $r$-circulant} that is obtained after applying the $circ(.)$ operator $r$ times.

For a permutation $\Pi$ of indices in $\{1,\ldots,n\}$, let $\PPP_{\Pi}$ be the corresponding $n \times n$ permutation matrix. Then, the {\em automorphism group} Aut$(\A)$ of an $n \times n$ matrix $\A$ is defined to be 
$$ \text{Aut}(\A)=\{\Pi\, | \, \PPP_{\Pi} \A \PPP_{\Pi}^{\top}=\A\}.$$
For a $G$-developed matrix $\A$,  
 if we permute the indices of $\A$ by the action of
multiplication by elements of $G$, then the elements of $G$ can be thought of as a set of permutations matrices that form a subgroup of $ \text{Aut}(\A)$. 
Hence, Aut$(\A) \geq G$ and it is easy to construct 
 examples where Aut$(\A) > G$. The set of all matrices whose automorphism group contains $G$ and entries are in $R$ is isomorphic to $R[G]$. This  follows by taking {$X=G$} on page 4 in~\cite{Iverson}.
 In particular, the products and integer linear combinations of circulant  ($r$-circulant) matrices is circulant ($r$-circulant). 


 There is an injection $\Psi$ of $\{0,1\}$ or $\{-1,1\}$
 $G$-arrays into $\mathbb{Z}[G]$ given by 
  $$\Psi(\aaa)=\sum_{g \in G} a_g g.$$ We say that the 
  $G$-array $\aaa$ {\em corresponds} to $A\in \mathbb{Z}[G]$ if 
  $A=\PPsia$. For a group ring element $\sum_{g \in G} a_g g$ corresponding to a $G$-array $(a_g)$ and an isomorphism $\Phi:G\rightarrow\Phi(G)$ we define $\Phi(A)$ to be
 $$\Phi(A)=\sum_{g \in G}a_g\Phi(g)=\sum_{\Phi(g) \in \Phi(G)}a_g\Phi(g).$$
 Throughout the paper, by abuse of notation,  if a set $H \subseteq G$ appears in a group ring equation, it is understood that  $H=\HHHH$. 
 Moreover, for  $A=\sum_{g \in G} a_g g\in \mathbb{Z}[G]$, we define $$\{A\}:=\{g \in G\, |\, a_g=1\}.$$  
  
The group ring elements that correspond to $G$-arrays are used to calculate the autocorrelation and cross-correlation functions of $G$-arrays, i.e., for a multiplicatively written group $G$, and $G$-arrays $\aaa$ and $\bb$ 
\begin{equation}\label{eqn:Cab}
 C_{\aaa,\bb}(t)=\text{coefficient of $t$ in }A\xoverline{B}^{(-1)},
 \end{equation} where $A=\PPsia,\, B=\PPsib$.

 A matrix $\M$ with entries in $\mathbb{R}$ is {\em symmetric (skew-symmetric)} if $\M=\M^{\top}$ ($\M=-\M^{\top}$).
Next, we define symmetric, skew-symmetric $G$-arrays, and skew-type matrices.
\begin{definition}\label{def:sym}
Let $G$ be a finite group with identity $e$. Let ${\boldsymbol m}=(m_{g})$ be a $\{0,1\}$ 
$G$-array and $M= \sum_{g \in G} m_{g} g$. Then, ${\boldsymbol m}$ or $M$ is  {\em symmetric} if  $M=M^{(-1)}$ and {\em skew-symmetric} if  $M+M^{(-1)}=G+e$ (implying $e\in \{M\}$) or $M+M^{(-1)}=G-e$ (implying $e \notin \{M\}$).
\end{definition}
The following  lemma shows that an isomorphism 
$\Phi:G\rightarrow\Phi(G)$ maps a symmetric (skew-symmetric) $G$-array to a symmetric (skew-symmetric)
$\Phi(G)$-array.
\begin{lemma}\label{lem:symskewiso}
Let $(a_g)$ be a symmetric (skew-symmetric) $G$-array.
Let $\Phi$ be
an isomorphism $\Phi:G\rightarrow\Phi(G)$. Then   
$\Phi((a_g))$ is a symmetric (skew-symmetric) $G$-array.
\end{lemma}
\begin{proof}
Let $A=\sum_{g \in G}a_gg\in \mathbb{Z}[G]$ and 
$\Phi$ be extended linearly to an isomorphism of $R[G]$ and
$R[\Phi(G)]$.  Then, $A=A^{(-1)}$ implies
 $\Phi(A)=\Phi(A^{(-1)})$ ($M+M^{(-1)}=G+e$ implies
 $\Phi(M)+\Phi(M^{(-1)})=\Phi(G)+\Phi(e)$ and $M+M^{(-1)}=G-e$ implies  $\Phi(M)+\Phi(M^{(-1)})=\Phi(G)-\Phi(e)$).
\end{proof}
The following  lemma  follows immediately from Definition~\ref{def:sym}.
\begin{lemma}\label{lem:symskewG}
Let $G$ be a finite group with identity $e$ and $M$ be the group ring element corresponding to ${\boldsymbol m}$. Then, a  $\{0,1\}$ $G$-array ${\boldsymbol m}$ is symmetric
(skew-symmetric) if and only if $\{M\}=\{M^{(-1)}\}$
($\{M\}\cup\{M^{(-1)}\}=G\backslash e$, $\{M\}\cap\{M^{(-1)}\}=\emptyset$ when $e \notin \{M\}$ and $\{M\}\cup\{M^{(-1)}\}=G$, 
$\{M\}\cap\{M^{(-1)}\}=e$ when $e \in \{M\}$).
\end{lemma}
A matrix $\M$ is  of {\em skew-type} if $\text{Diag}(\M)=\I$ and
$(\M-\text{Diag}(\M))$ is skew-symmetric, where $\text{Diag}(\M)$ is the diagonal matrix obtained from $\M$ by replacing each non-diagonal entry of $\M$ with $0$. 
Now, it is plain to see the following lemma.
\begin{lemma}
Let ${\boldsymbol m}=(m_{g})$ be a  $\{0,1\}$ $G$-array with $m_e=1$. 
Let $\M$ be the group developed matrix obtained by using $m_{e,g}=m_g$ as its first row and $\J_{|G|\times|G|}$ be the $|G|\times|G|$ matrix of all $1$s.
Then $2\M-\J_{|G|\times|G|}$ is symmetric (skew-type) if and only if $M=\sum_{g \in G}m_gg$ is symmetric (skew-symmetric).
\end{lemma}

\subsection{Legendre $G$-array pairs}
First, we define Legendre $G$-array pairs.
\begin{definition} \label{DefLegendrePair}
Let $G$ be a multiplicatively written finite abelian group with $|G|=n$. 
Then, a pair of $\{-1,1\}$  $G$-arrays  $(\aaa=(a_g),\bb=(b_g))$ form a {\em Legendre $G$-array pair} if $ \sum_{g \in G} a_g =\sum_{g \in G}b_g$ and 
\begin{equation}\label{eqn:A+B}
AA^{(-1)}+BB^{(-1)}=
\left(|A|+|B|\right)(1)-2(G-1), 
\end{equation}
where $A$ and $B$ are the group ring elements
associated with $\aaa$ and $\bb$.
\end{definition} 
By applying the principal character to the group ring equation~(\ref{eqn:A+B}) we get
\begin{align}
	\chi_0 \left( AA^{(-1)}+BB^{(-1)} \right) =&\,\, \chi_0 \left( (|A|+|B|)(1)-2(G-1) \right) \nonumber\\
	\chi_0(A)^2 + \chi_0(B)^2 =&\,\, 2n-2(n-1) \nonumber\\
	a^2 + b^2 =&\,\, 2. \label{eqn:a2b2}
\end{align}
This equation implies that $a=b\in\{-1,1\}$, where 
$a=\sum_{g \in G} a_g=b=\sum_{g \in G}b_g$. Thus $|G|=n$ must be odd for a Legendre $G$-array pair to exist. Hence, each $G$-array in a Legendre $G$-array pair must be balanced. 

By equations~(\ref{eqn:EqnACF}), (\ref{eqn:A+B}), and~(\ref{eqn:a2b2}) we get the following  definition of Legendre $\{0,1\}$  
$G$-array pairs.
 \begin{definition} \label{DefLegendrePair2}
 Let $G$ be a multiplicatively written finite abelian group.
A pair of $\{0,1\}$  $G$-arrays  $(\aaa=(a_g),\bb=(b_g))$ form a {\em Legendre $G$-array pair} if $ \sum_{g \in G} a_g =\sum_{g \in G}b_g$, and 
\begin{equation*}\label{eqn:A+B2}
AA^{(-1)}+BB^{(-1)}=\begin{cases}
2\left(\frac{|G|+1}{2}\right)(1)+\frac{|G|+1}{2}(G-1) \quad \text{if $\sum_{g \in G} a_g=\sum_{g \in G} b_g=\frac{|G|+1}{2}$,}\\
2\left(\frac{|G|-1}{2}\right)(1)+\frac{|G|-3}{2}(G-1) \quad \text{if $\sum_{g \in G} a_g=\sum_{g \in G} b_g=\frac{|G|-1}{2}$.}
\end{cases}  
\end{equation*}
\end{definition}
The following lemma is plain to prove.
\begin{lemma}\label{iso}
 Let $\Phi:G\rightarrow \Phi(G)$ be an isomorphism. Then, $((a_g),(b_g))$ is a Legendre $\{-1,1\}$  ($\{0,1\}$)  
 $G$-array pair if and only if $(\Phi((a_g)), \Phi((b_{g})))$ is a Legendre $\{-1,1\}$  ($\{0,1\}$)  $\Phi(G)$-array pair. Hence, whenever we construct a Legendre $\{-1,1\}$ ($\{0,1\}$)  $G$-array pair  we have also constructed a Legendre $\{-1,1\}$ ($\{0,1\}$) $\Phi(G)$-array pair.
 \end{lemma}
 The following well-known theorem connects supplementary difference sets in finite abelian groups and Legendre $G$-array pairs.
 \begin{theorem}\label{thm:Legendre}
 Let $G$ be an abelian group of order $n$.
 Let $(\aaa,\bb)$ be a $\{0,1\}$ or $\{-1,1\}$ $G$-array pair and $(M,N)$ be the subsets of $G$ such that $M=\{g\in G\,|\,a_g=1\}$ and $N=\{g\in G\,|\,b_g=1\}$. Then, $(M,N)$ is a $2$-SDS$(n;(n+1)/2,(n+1)/2;(n+1)/2)$ or a $2$-SDS$(n;(n-1)/2,(n-1)/2;(n-3)/2)$ if and only if $(\aaa,\bb)$ is a 
 Legendre $G$-array pair.
\end{theorem}
\begin{proof}
If $G$ is written multiplicatively, then the result follows by comparing Definition~\ref{Def:SDS} for $f=2$ 
to Definition~\ref{DefLegendrePair} 
 (Definition~\ref{DefLegendrePair2})
 for $\{-1,1\}$ ($\{0,1\}$) $G$-arrays. 
\end{proof}
 It is conjectured that  a Legendre $\mathbb{Z}_n$-array pair exists for all odd $n$ \cite{SeberryGLPairs}. 
 A Legendre $\mathbb{Z}_n$-array pair is known to exist when:
\begin{itemize}
	\item $n$ is a prime, see~\cite{SeberryGLPairs}; 
	\item $2n+1$ is a prime power (Szekeres, \cite{Wallis}); 
	\item $n=2^m-1$ for $m \geq 2$, see~\cite{Schroeder}; 
	\item $n=p_1(p_1+2)$, with $p_2=p_1+2$, where $p_1,p_2$  are odd primes~\cite{twin}. 
\end{itemize}
Currently, $n=77$ is the smallest $n$  for which no Legendre $\mathbb{Z}_n$-array pair is known. 

An $N\times N$ {\em Hadamard matrix}, $\HH$, is a $\pm1$ matrix such that $\HH\HH^{\top}=N\I_N$ where $\I_N$ is the identity matrix of order $N$. The following theorem showing  that the existence of a Legendre $G$-array pair implies the existence of a $(2|G|+2)\times(2|G|+2)$ Hadamard matrix is well-known.
\begin{theorem}\label{thm:bordered}
  Let $(\aaa,\bb)$ be a Legendre $\{-1,1\}$ $G$-array pair, with $|G|=n$
  such that $\sum_{g \in G}a_g=\sum_{g \in G}b_g=1$. Let $(\aaa,\bb)$ be developed into $G$ indexed $n \times n$ matrices 
  $\A$ and $\B$ by taking $\aaa$ and $\bb$ as the first row of 
  $\A$ and $\B$ respectively. Let
$$\HH_{sym}=\begin{bmatrix}
- & - & + & \cdots & + & + & \cdots & + \\ 
- & + & + & \cdots & + & - & \cdots & - \\ 
+ & + &  &  &  &  &  &  \\ 
\vdots & \vdots &  & \A &  &  & \B &  \\ 
+ & + &  &  &  &  &  &  \\ 
+ & - &  &  &  &  &  &  \\ 
\vdots & \vdots &  & \B^{\top} &  &  & -\A^{\top} &  \\ 
+ & - &  &  &  &  &  & 
\end{bmatrix}$$
and
$$\HH_{skew}=\begin{bmatrix}
+ & + & + & \cdots & + & + & \cdots & + \\ 
- & + & + & \cdots & + & - & \cdots & - \\ 
- & - &   &  &  &  &  &  \\ 
\vdots & \vdots &  & \A &  &  & \B &  \\ 
- & - &  &  &  &  &  &  \\ 
- & + &  & &  &  &  &  \\ 
\vdots & \vdots &  & -\B^{\top} &  &  & \A^{\top} &  \\ 
- & + &  &  &  &  &  & 
\end{bmatrix}.$$
Then, both $\HH_{sym}$  and $\HH_{skew}$ are Hadamard matrices.
Moreover, $\HH_{sym}$  ($\HH_{skew}$) is symmetric (skew-type) Hadamard matrix if and only if  $\aaa$ is symmetric (skew-symmetric).
\end{theorem}
\begin{proof}
Let $e$ be the identity element in $G$.
The matrix $\HH_{sym}$ ($\HH_{skew}$) is a Hadamard matrix if and only if $C_{\aaa}(t)+C_{\bb}(t)=-2$ for all $t \in G\setminus \{e\}$.
Then, by using equation~(\ref{eqn:Cab}) for $C_{\aaa,\aaa}$
and $C_{\bb,\bb}$, we get $C_{\aaa}(t)+C_{\bb}(t)=-2$ for all $t \in G\setminus \{e\}$ if and only if $(\aaa,\bb)$ is a Legendre $G$-array pair.
The matrix $\HH_{sym}$ ($\HH_{skew}$) is symmetric (skew-type) if
 and only if $\A$ is symmetric (skew-type). The result now follows as $\A$ is symmetric (skew-type) if and only if $\aaa$ is symmetric (skew-symmetric).
\end{proof}

Consider the action of the group $\ZZZS$
on the group $\ZZZ$ defined by 
$$  ((a_1,b_1),\ldots,(a_r,b_r))(g_1,\ldots,g_r)=(b_1g_1+a_1,
 \ldots,b_rg_r+a_r) $$
 if the group $\ZZZ$ is written additively,  and
 $$  ((a_1,b_1),\ldots,(a_r,b_r))(g_1,\ldots,g_r)=(g_1^{b_1}a_1, \ldots,g_r^{b_r}a_r) $$
 if $\ZZZ$ is written multiplicatively,
where $\mathbb{Z}^*_{s_i}$ is the multiplicative group of the ring $\mathbb{Z}_{s_i}$ and $\rtimes$ is the {\em semidirect product} as defined 
in~\cite[p. 167]{RotmanBook}. This group action can be extended linearly to 
$\mathbb{Z}[G]$. Then, 
$\ZZZS$ 
acts on a $\ZZZ$-array, and two $\ZZZ$-arrays are called {\em equivalent} if one can be
 obtained from the other by applying the elements of the group 
$$\ZZZS.$$ It is well-known that if $\aaa$ and $\aaa'$ are equivalent $\ZZZ$-arrays then  $\aaa$ and $\aaa'$ have the same  peak correlation and the same correlation energy.
We call two Legendre pairs $(\aaa,\bb)$ and $(\aaa',\bb')$ {\em equivalent}
 if $\{\aaa,\bb\}=\{\tau\aaa',\beta\bb'\}$, where 
 $\tau=((\tau_1,\tau_1^*),\ldots,(\tau_r,\tau_r^*))$, $\beta=((\beta_1,\beta_1^*),\ldots,(\beta_r,\beta_r^*))$ such that $\beta_i,\tau_i \in \mathbb{Z}_{s_i}$, $\beta^*_i,\tau^*_i \in \mathbb{Z}^*_{s_i}$ and $\tau_i^*=\pm\beta_i^*$ for $i=1,\ldots,r$~\cite{SeberryGLPairs}.
If $(\aaa, \bb)$ is a Legendre $\ZZZ$-array pair,
and $(\aaa',\bb')$ is equivalent to $(\aaa,\bb)$, 
 then $(\aaa', \bb')$ is also a Legendre $\ZZZ$-array pair.
  
 The following lemma determines exactly which subgroup of 
$\ZZZS$ preserves symmetry 
 (skew-symmetry) of a symmetric (skew-symmetric) $\ZZZ$-array.
 \begin{lemma}\label{lem:symskew}
 The group $\ZZZSZ$ preserves the symmetry (skew-symmetry) of a symmetric (skew-symmetric) $\{-1,1\}$ or $\{0,1\}$ $\ZZZ$-array.
 \end{lemma}
\begin{proof}
Let $\aaa$ be a symmetric (skew-symmetric) $\{-1,1\}$ or $\{0,1\}$ $\ZZZ$-array, and 
$A\subset\ZZZ$ be the set of $1$ indices of $\aaa$. 
Then,  $A=-A$ 
(($A\cup -A=\ZZZ\backslash \{0\}$ and $A\cap-A=\emptyset$) or 
($A\cup -A=\ZZZ $ and $A\cap-A=0$)) implies for any $((0,\beta^*_1),\ldots,(0,\beta^*_r))\in \ZZZSZ$ 
$$((0,\beta^*_1),\ldots,(0,\beta^*_r))A=
-((0,\beta^*_1),\ldots,(0,\beta^*_r))A$$ 
\begin{eqnarray*}
(((0,\beta^*_1),\ldots,(0,\beta^*_r))A\cap-((0,\beta^*_1),\ldots,(0,\beta^*_r))A&=&\emptyset \quad \text{or}\\ ((0,\beta^*_1),\ldots,(0,\beta^*_r))A\cap-((0,\beta^*_1),\ldots,(0,\beta^*_r))A&=&0).
\end{eqnarray*}
\end{proof}
In general, Lemma~\ref{lem:symskew} can not be improved as it is easy to construct a symmetric (skew-symmetric) $\mathbb{Z}_s$-array whose symmetry (skew-symmetry) is not preserved by any circulant shifts other than the $0$ shift.

 We fix some notation
that will be used in the rest of the paper.
Let $q=p^m$ for some prime $p$ and positive integer $m$. Let $\mathbb{F}_q$ be the finite field with $q$ elements and $\mathbb{F}_q^*=\langle \alpha\rangle$
be the multiplicative group of $\mathbb{F}_q$, where $\alpha$ is a generator for $\mathbb{F}_q^*$. Let $C_0^{(d,q,\alpha)}=\langle\alpha^d\rangle$ be the multiplicative group generated by $\alpha^d$ in the finite field $\mathbb{F}_q$, where $d$ divides $q-1$. Observe that $C_0^{(d,q,\alpha)}$ does not depend on $\alpha$. Let $C_i^{(d,q,\alpha)}=\alpha^i C_0^{(d,q,\alpha)}$ for $i=0,1,\ldots,d-1$, where $C_i^{(d,q,\alpha)}$ are called {\em cyclotomic classes of order $d$}, see~\cite{StorerBook}.
We will denote  $C_i^{(d,q,\alpha)}$ with $C_i^{d}$
when there is no need to specify $q$ and $\alpha$. The labeling of $C_1^{(d,q,\alpha)}, \ldots,C_{d-1}^{(d,q,\alpha)}$ depends
on $\alpha$, but taking a different choice of primitive root just 
permutes $C_1^{(d,q,\alpha)}, \ldots,C_{d-1}^{(d,q,\alpha)}$.
\subsection{Infinite families of perfect $G$-arrays}\label{sec:families}
First, we survey several known infinite families of perfect $G$-arrays.\\
\noindent \textbf{The Sidelnikov-Lempel-Cohn-Eastman 
$\mathbb{Z}_{q-1}$-arrays:}\\
 Let  $$S=\{\alpha^{2i+1} -1\}_{i=0}^{\frac{q-1}{2}-1}. $$ 
 Let $\aaa$ be a $\{-1,1\}$ or $\{0,1\}$ $(q-1)\times 1$
  vector such that 
$$ a_i=  1  \quad {\rm if}\ \alpha^i\in S.$$
 Then, $\aaa$ is the Sidelnikov-Lempel-Cohn-Eastman 
 $\mathbb{Z}_{q-1}$-array. The Sidelnikov-Lempel-Cohn-Eastman $\mathbb{Z}_{q-1}$-array is always balanced. However, it is perfect if and only if $q=p^m \equiv 3\; (\text{mod} \;4)$, see \cite{Lempel} and \cite{Sidelnikov}. 

\noindent \textbf{The Ding-Helleseth-Martinsen $\mathbb{Z}_{2}\times \mathbb{Z}_p^{m}$-arrays:}\\
\noindent
 Let $p\equiv1$ (mod $4$) and $p^m=s^2+4t^2$, where $s^2=1$ or $t^2=1$. 
  Let $q=p^m\equiv 5 $ $\; ( \text{mod }  8)$, or equivalently, let  
 $p\equiv 5 \; (\text{mod }  8)$ and $m$ be odd.
 Let
 $C_{i,j,\ell}=\left(C^4_i \cup C^4_j, C^4_j \cup C^4_\ell \right)$ for
 $\{i,j,l\} \subset \{0,1,2,3\}$, where $i,j,l$ are  distinct integers.
Let \begin{align*}
(A_1,B_1)=(C^4_0 \cup C^4_1,C^4_1 \cup C^4_3),\,\, \,(A_2,B_2)=(C^4_0 \cup C^4_2,  C^4_2\cup C^4_3) \quad \text{if } \,  t^2=1,\\
(A_3,B_3)=(C^4_0 \cup C^4_1, C^4_0 \cup C^4_3) \quad  \text{if } 
 s^2=1.
 \end{align*} 
 Identify the elements of the finite field $\mathbb{F}_{p^m}$ with its additive group $\mathbb{Z}_p^m$, and let $\langle \omega \rangle=
 \Theta(\mathbb{Z}_2)$, where $\Theta$ be the isomorphism in Section~\ref{sec:introcorrelations}.
 We now use the group $\Theta(\mathbb{Z}_{2}\times \mathbb{Z}_p^{m})$ as an indexing set. 
For each $i\in \{1,2,3\}$, let the equivalence class $i$ Ding-Helleseth-Martinsen $\{-1,1\}$ or $\{0,1\}$ $\Theta(\mathbb{Z}_{2}\times \mathbb{Z}_p^{m})$-array be such that $\Theta(A_i) \cup \Theta(B_i)\omega$ is the set of $1$ indices of the array.
Then, each equivalence class of Ding-Helleseth-Martinsen $\Theta(\mathbb{Z}_{2}\times \mathbb{Z}_p^{m})$-array is almost balanced, and equivalence class $3$ is always perfect, see Theorem 2 in~\cite{DingHelleseth}. In Section~\ref{open} we determine exactly when each of the equivalence class $1$ and  $2$ Ding-Helleseth-Martinsen  $\Theta(\mathbb{Z}_{2}\times \mathbb{Z}_p^{m})$-array is
perfect. This solves one of the two open problems posed in~\cite{Ding2008}.
  Finally,  each equivalence class of $\{-1,1\}$ or $\{0,1\}$ Ding-Helleseth-Martinsen
 $\Theta(\mathbb{Z}_{2}\times \mathbb{Z}_p^{m})$-array $(a_g)$ is used to construct the corresponding  $\{-1,1\}$ or $\{0,1\}$
  Ding-Helleseth-Martinsen
 $\mathbb{Z}_{2}\times \mathbb{Z}_p^{m}$-array as  $\Theta^{-1}((a_g))$.
 \subsection{Infinite families of  Legendre $G$-array pairs}
 Now, we survey several known infinite families of Legendre $G$-array pairs.\\
\noindent \textbf{The Yamada $\mathbb{Z}_{(q-1)/2}$-array
  pairs:}\\
 Let $q=p^m \equiv 3 \; (\hspace{-.25cm}\mod 4)$. 
  Let $$M=\{a:\alpha^{2a}+1 \in C_0^2\},$$ and $$N=\{a:\alpha^{2a}-1 \in C_0^2\}.$$ 
  Then the pair $(M,N)$ is a $2$-SDS$((q-1)/2;(q-3)/4,(q-3)/4;(q-7)/4)$
in $\mathbb{Z}_{(q-1)/2}$.
  Let $\mathbb{Z}_{(q-1)/2}$ index the arrays $\aaa,\bb$, and  $(M, N)$ be the sets of $1$ indices 
  of $(\aaa,\bb)$.
  Then the $\{-1,1\}$ or $\{0,1\}$
   $\mathbb{Z}_{(q-1)/2}$-array pair $(\aaa,\bb)$ is called  a {\em Yamada $\mathbb{Z}_{(q-1)/2}$-array pair}, see~\cite{Yamada}. The $\mathbb{Z}_{(q-1)/2}$-array $\aaa$ is symmetric and $\bb$ is skew-symmetric.

\noindent \textbf{The Szekeres $\mathbb{Z}^m_{p}$-array  pairs:}\\
 Let $q=p^m\equiv 5 \; (\hspace{-.25cm}\mod  8)$, or equivalently, let  
 $p\equiv 5 \; (\hspace{-.25cm}\mod  8)$ and $m$ be odd. Let $$A=C^4_0 \cup C^4_1, \; B=C^4_0 \cup C^4_3.$$ Then the pair $(A,B)$ is a $2$-SDS$(p^m;(q-1)/2,(q-1)/2;(q-3)/2)$
in $\mathbb{Z}^m_{p}$.  Let $\mathbb{Z}^m_{p}$ index the arrays $\aaa,\bb $, and  $(A, B)$ be the sets of $1$ indices 
of $(\aaa,\bb)$. 
Then the $\{-1,1\}$ or $\{0,1\}$ $\mathbb{Z}^m_{p}$-array pair $(\aaa,\bb)$ is called  a {\em Szekeres $\mathbb{Z}^m_{p}$-array pair},  see~\cite{SzekeresCompDS}. 
 Both $\aaa$ and $\bb$ are skew-symmetric. 
 
 \noindent \textbf{The Szekeres-Whiteman $\mathbb{Z}^m_{p}$-array  
 pairs:}\\
  Let $q=p^m$,  $p\equiv 5 \; (\hspace{-.25cm}\mod  8)$ and $m $ be even with $m\geq 2$.
  Let $$A=C^8_0 \cup C^8_1 \cup C^8_2 \cup C^8_3, \; 
  B=C^8_0 \cup C^8_1 \cup C^8_6 \cup C^8_7.$$ 
Then the pair $(A,B)$ is  a $2$-SDS$(p^m;(q-1)/2,(q-1)/2;(q-3)/2)$ in $\mathbb{Z}^m_{p}$.
  Let $\mathbb{Z}^m_{p}$ index the arrays $\aaa,\bb $, and  $(A, B)$ be the sets of $1$ indices of
   $(\aaa,\bb)$. 
  Then the $\{-1,1\}$ or $\{0,1\}$ $\mathbb{Z}^m_{p}$-array pair $(\aaa,\bb)$ is called  a {\em Szekeres-Whiteman $\mathbb{Z}^m_{p}$-array pair}.
  Szekeres~\cite{SzekeresCompDS} proved that a  Szekeres-Whiteman $\mathbb{Z}^m_{p}$-array pair is a Legendre $\mathbb{Z}^m_{p}$-array pair, while Whiteman~\cite{Whiteman} independently showed this result however only for 
  $m\equiv2 \,\,  \text{(mod $4$)}$. It is easy to see that both $\aaa$ and $\bb$ are skew-symmetric. 
  
  \noindent \textbf{The Paley $\mathbb{Z}^m_{p}$-array  pairs:}\\
 Let \begin{align*}
 A=C^2_0,  \; B=C^2_0 \quad &\text{if $p^m\equiv3$ (mod $4$),}\\ 
 A=C^2_1,  \; B=C^2_0 \quad &\text{if $p^m\equiv1$ (mod $4$).}
 \end{align*}
 Then the pair $(A,B)$ is a $2$-SDS$(p^m;(q-1)/2,(q-1)/2;(q-3)/2)$
in $\mathbb{Z}^m_{p}$, see~\cite{SeberryGLPairs}.  Let $\mathbb{Z}^m_{p}$ index the arrays $\aaa,\bb $, and   
$(A, B)$ be the sets of $1$ indices 
of $(\aaa,\bb)$.
 Then the $\{-1,1\}$ or $\{0,1\}$ $\mathbb{Z}^m_{p}$-array pair $(\aaa,\bb)$ is called  a {\em Paley $\mathbb{Z}^m_{p}$-array pair}.  Both $\aaa$ and $\bb$ are 
 skew-symmetric if $p^m\equiv3$ (mod $4$) and symmetric otherwise. 
 
\noindent \textbf{The Baumert $\mathbb{Z}^{m_1}_{p_1}\times \mathbb{Z}^{m_2}_{p_2}$-array  pairs:}\\
Let $p_1^{m_1}+2= p_2^{m_2}$, where $p_1,p_2$ are odd primes and $m_1,m_2$ are positive integers. Let $q_1=p_1^{m_1}$, $q_2=p_2^{m_2}$,
and
 \begin{align*}
 A=\left(C^{(2,q_1)}_0\times C^{(2,q_2)}_0\right)\bigcup\left(C^{(2,q_1)}_1\times C^{(2,q_2)}_1\right)\bigcup \left(\FF_{q_1}\times \{0\}\right),  \;\; B=A.
 \end{align*}
 Since $A$ is a DS$(q_1(q_1+2),(q_1^2+2q_1-1)/2,(q_1-1)(q_1+3)/4)$  in $\mathbb{Z}^{m_1}_{p_1}\times \mathbb{Z}^{m_2}_{p_2}$~\cite{twin}, the pair $(A,B)$ is a $2$-SDS$(p_1^{m_1}p_2^{m_2};(q_1^2+2q_1-1)/2,(q_1^2+2q_1-1)/2;(q_1-1)(q_1+3)/2)$
in $\mathbb{Z}^{m_1}_{p_1}\times \mathbb{Z}^{m_2}_{p_2}$.
Let $\mathbb{Z}^{m_1}_{p_1}\times \mathbb{Z}^{m_2}_{p_2}$ index the arrays $\aaa, \bb$, and  $(A, B)$ be the sets of $1$ indices
  of $(\aaa,\bb)$. 
 Then the $\{-1,1\}$ or $\{0,1\}$
  $\mathbb{Z}^{m_1}_{p_1}\times \mathbb{Z}^{m_2}_{p_2}$-
 array pair $(\aaa,\bb)$ is called  a {\em Baumert $\mathbb{Z}^{m_1}_{p_1}\times \mathbb{Z}^{m_2}_{p_2}$-array pair}.  
 Both $\aaa$ and $\bb$ are 
 neither symmetric nor skew-symmetric.
\section{Results} \label{sec:results}
\subsection{Yamada-Pott $G$-array pairs}
 Yamada-Pott $G$-array pairs first appeared in~\cite{Yamada} and later in~\cite{PottBook}. A Yamada-Pott $\{0,1\}$ $G$-array pair is a Legendre $\{0,1\}$ $G$-array pair with the added properties that one $G$-array  is symmetric and the other is skew-symmetric. In group ring notation we have the following definition.
\begin{definition}\label{yamadapott}
Let $G$ be a finite abelian group written multiplicatively. A Legendre $\{0,1\}$ $G$-array pair $(\aaa,\bb)$  with 
$A = \PPsia$ and $B = \PPsib$ 
 is a {\em Yamada-Pott  $\{0,1\}$ $G$-array pair} if  $|A|=|B|$  and:
\begin{enumerate}
	\item $A=A^{(-1)}$; \label{pottla}
	\item $B+B^{(-1)}=G+1$ (implying $1\in \{B\}$) or $B+B^{(-1)}=G-1$
	(implying $1 \notin \{B\}$)\label{pottsan}
\end{enumerate}
 are satisfied.
\end{definition}
The following lemma is plain to prove.
\begin{lemma}~\label{isomorphism}
Let  $\Phi:G\rightarrow \Phi(G)$ be an isomorphism. Then, $((a_g),(b_g))$ is a Yamada-Pott  $\{0,1\}$ $G$-array pair if and only if $(\Phi((a_g)),\Phi((b_g)))$ is a Yamada-Pott  $\{0,1\}$ $\Phi(G)$-array pair. 
\end{lemma}
 By Lemma~\ref{isomorphism}, whenever we construct a  
 Yamada-Pott  $\{0,1\}$  $G$-array pair,
 we have also constructed a Yamada-Pott  $\{0,1\}$  
 $\Phi(G)$-array pair. 
The following theorem implies that the existence of a Yamada-Pott $\{0,1\}$ 
$\mathbb{Z}_u$-array pair   implies the existence of a perfect  $\{0,1\}$ $\mathbb{Z}_{2u}$-array.

\begin{theorem} \label{ThmYPPairToADS}
Let $H$ be an abelian group  with $|H|=u$ written multiplicatively and $(A, B)$ be a Yamada-Pott $\{0,1\}$ $H$-array pair. Let $S=A+\omega B$ and $G=\langle\omega\rangle H$, where $\omega^2=1$, $\omega\neq1$, and $\omega h=h\omega$ for all $h \in H$. Then, $1\in \{B\}$ implies $\{S\}$ is an ADS$(2u, u+1, (u+1)/2,  (u+3)/2 )$, and  $1\notin \{B\}$ implies $\{S\}$ is an
ADS$\left(2u, u-1, (u-1)/2, (u-3)/2 \right)$ in $G$. 
In either case, $\s$ is an  
almost balanced perfect $\{0,1\}$ $G$-array with 
$G\cong \mathbb{Z}_2\times\Theta(H)\cong \mathbb{Z}_2\times H$, where $\s$ is the $G$-array that corresponds to the group ring element $S$. 
\end{theorem}
\begin{proof}
Since $|S|=u\pm1$ then $\s$ is almost balanced, and by the definition of a Yamada-Pott $\{0,1\}$ $H$-array pair  
\begin{align*} 
 AA^{(-1)} + BB^{(-1)} =&\,\, (u \pm 1)(1) + \lambda (H-1) \\
 A=&\,\,A^{(-1)}\\
 B+B^{(-1)}=&\,\,H \pm 1,
 \end{align*}
 where 
 $$
 \lambda=
 \begin{cases}
\frac{u+1}{2} \quad \quad\text{if $|A|=|B|= \frac{u+1}{2}$},\\
 \frac{u-3}{2} \quad \quad\text{if $|A|=|B|= \frac{u-1}{2}$}.
 \end{cases}
 $$
Now, 
\begin{align*}
\sum_{t \in G}C_{\s}(t)t=SS^{(-1)} =&\,\, \left( A+\omega B \right) \left( A+\omega B \right)^{(-1)} \\
=& \,\,\left( A+\omega B \right) \left( A^{(-1)}+\omega B^{(-1)} \right) \\
=&\,\, AA^{(-1)} + BB^{(-1)} + \omega \left( AB^{(-1)} + BA^{(-1)} \right) \\
=&\,\, (u\pm1)(1) + \lambda (H-1) + \omega \left( AB^{(-1)} + BA^{(-1)} \right) \\
=&\, \, (u\pm1)(1) + \lambda (H-1) + \omega \left( AB^{(-1)} + BA \right) \\
=&\,\, (u\pm1)(1) + \lambda (H-1) + \omega A \left( B^{(-1)} + B \right) \\
=& \,\,(u\pm1)(1) + \lambda (H-1) + \omega A \left( H \pm1  \right) \\
=&\,\, (u\pm1)(1) + \lambda (H-1) + \omega \left( |A|H \pm A \right) \\
=&\,\, (u\pm1)(1) + \lambda (H-1) + \omega \left( \left( \frac{u\pm1}{2} \right)H \pm A \right) \\
	=&\begin{cases}\frac{u+1}{2}(1) + \lambda G +A\omega \, \hspace{.1cm} \, \, \, \quad \quad \quad \, \text{if $|A|=|B|=\frac{u+1}{2}$,}\\
	\frac{u+1}{2}(1) + \lambda G +(H-A)\omega \quad \, \text{if $|A|=|B|=\frac{u-1}{2}$,}\end{cases}
\end{align*}
where we used the group ring equation $G=H+H\omega$. 
This shows that for $t \neq 1$ the autocorrelation function of $\s$ has the following form 
\begin{equation}\label{eqn:auto}
C_{\s}(t)=
\begin{cases}
\frac{u+1}{2} \text{ or }\frac{u+1}{2}+1\quad \text{if $|A|=|B|=\frac{u+1}{2}$},\\
\frac{u-3}{2} \text{ or }\frac{u-3}{2}+1\quad \text{if $|A|=|B|=\frac{u-1}{2}$}.
\end{cases}
\end{equation}
Thus, by equations~(\ref{perfect01}),~(\ref{almostbalanced01}), 
and~(\ref{eqn:auto}), $\s$ is an almost balanced and perfect $\{0,1\}$ $G$-array. 
\end{proof}

The following are a few remarks concerning Theorem~\ref{ThmYPPairToADS}.
\begin{enumerate}
\item The equation $AB^{(-1)} + BA = A(B^{(-1)} + B)$ in the proof of Theorem~\ref{ThmYPPairToADS}
 is allowed only when
$G$ is abelian. All other steps in the proof would hold for arbitrary finite groups.
	\item The converse to Theorem~\ref{ThmYPPairToADS} is not true. That is, having a balanced and perfect $\{0,1\}$ $G=\mathbb{Z}_2 \times H$-array  
	does not guarantee the existence of a Yamada-Pott 
	$\{0,1\}$ $H$-array pair via reversing the construction in Theorem~\ref{ThmYPPairToADS}. For example, {\scriptsize $$\s=( 1,  1,  0,  0,  1,  1,  1, 1,  0,  0,  1,  0,  1,  1,  0,  0,  0,  0,  0,  0,  1,  0,  1, 0,  1,  1,  1,  0,  1,  0,  0, 1,  0,  1,  1,  1,  0,  0 )^{\top}$$}\noindent is a balanced and perfect $\{0,1\}$ $\mathbb{Z}_{38}$-array obtained by applying the $a_g \rightarrow \frac{a_g+1}{2}$ transformation to the  $\{-1,1\}$ $\mathbb{Z}_{38}$-array in~\cite{ArasuLittle}.  
	Let $S=\sum_{i \in \mathbb{Z}_{38}}s_ii$. 
Now, $\mathbb{Z}_{38} \cong \mathbb{Z}_{2}	\times \mathbb{Z}_{19}$ via the map $\phi(i)  = (i\, \, (\text{mod } 2),\, (i\, \, (\text{mod } 19))$. Let $\hat{S}=\sum_{i \in \mathbb{Z}_{38}}
s_{\phi(i)}\phi(i)=\sum_{i' \in \mathbb{Z}_2\times \mathbb{Z}_{19}}s_{i'}i'$ be the group ring element corresponding to $(\Phi(s_g))$ in $\mathbb{Z}[\mathbb{Z}_2\times \mathbb{Z}_{19}]$. 
 Write $\hat{S}=\hat{A}+\hat{B}$, where $\hat{A}=\sum_{i' \in \{0\}\times \mathbb{Z}_{19}}s_{i'}i'$  and $\hat{B}=\sum_{i' \in \{1\}\times \mathbb{Z}_{19}}s_{i'}i'$. Let $\pi:\mathbb{Z}_2\times \mathbb{Z}_{19} \rightarrow \mathbb{Z}_{19}$ be the projection map $\pi((x,y))=y$. Let $A=\sum_{i' \in \{0\}\times \mathbb{Z}_{19}}s_{\pi(i')}\pi(i')$ and $B=\sum_{i' \in \{1\}\times \mathbb{Z}_{19}}s_{\pi(i')}\pi(i')$. Let  $\aaa, \bb$ be the $\{0,1\}$ $\mathbb{Z}_{19}$-arrays corresponding to $A$ and $B$. Then, 
{\scriptsize $$\aaa=(1, 1, 0, 1, 1, 1, 1, 1,0, 1, 1, 0, 1, 0, 0,
 1, 0, 0, 0)^{\top},
 $$
$$\bb=(0, 1, 0, 0, 0, 1, 1, 1, 0, 0, 0, 0, 1, 1, 1, 0, 1, 0, 0)^{\top},$$}\noindent and 
$(\aaa,\bb)$ fails all of the Yamada-Pott $\{0,1\}$ $\mathbb{Z}_{19}$-array pair conditions, i.e. none of $\aaa$ and $\bb$ is symmetric or skew-symmetric and $|A|\neq |B|$. 
Observe that the isomorphism $\Theta: \mathbb{Z}_2 \times \mathbb{Z}_{19}\rightarrow \langle \omega \rangle H$ maps $(\Phi(s_g))$ to the $\langle \omega \rangle H$-array in Theorem~\ref{ThmYPPairToADS}, where $H$ is a cyclic group of order $19$. Hence, by Lemma~\ref{isomorphism}, reversing the construction in Theorem~\ref{ThmYPPairToADS} 
in this case does not produce a Yamada-Pott $\{0,1\}$ $H$-array pair.
In fact, by an exhaustive computer search, we showed that  no 
Yamada-Pott $\{0,1\}$ $\mathbb{Z}_{19}$-pair exists. Similarly, an exhaustive computer search proved that no
 Yamada-Pott  $\{0,1\}$ $\mathbb{Z}_{17}$-array pair exists. However, a balanced and perfect $\{0,1\}$ $\mathbb{Z}_{34}$-array exists as the Ding-Helleseth-Martinsen class 3 $\mathbb{Z}_2\times \mathbb{Z}_{17}$-array.
	\item There are families of balanced, $\{0,1\}$ $\mathbb{Z}_{2u}$-arrays with perfect autocorrelations that can be used to construct  
	Yamada-Pott $\{0,1\}$ $\mathbb{Z}_{u}$-array pairs or Szekeres $\{0,1\}$ $\mathbb{Z}_{u}$-array pairs, see Theorems~\ref{ThmDHM_s1_SzekeresPair} and~\ref{ThmDHM_t1_SzekeresPair}. 
	
	\item When $|A|=|B|=(u+1)/2$, the smaller (larger) correlation value appears at the elements of $H \cup (H-A)\omega$ ($A\omega$).
		\item When $|A|=|B|=(u-1)/2$, the smaller (larger) correlation value appears at the elements of $H \cup A\omega$ ($(H-A)\omega$).
\end{enumerate}
\begin{theorem}
Replacing $A$ with  $H-A$ or $B$ with $H-B$ in Theorem~\ref{ThmYPPairToADS} does not alter the 
Yamada-Pott $\{0,1\}$ $H$-array pair properties~1 and 2, and yields a perfect and balanced $\{0,1\}$ $\langle\omega\rangle H$-array.
\end{theorem}
\begin{proof}
	Let $G=\langle\omega\rangle H$ and $(A,B)$ be a
	 Yamada-Pott $\{0,1\}$ $H$-array pair. Let $S'=(H-A+\omega B)$.
Let $\s'$ be the $\{0,1\}$ $G$-array that corresponds to $S'$.	First, $\s'$ is balanced as $$|S'|=|H-A|+|B|=u-\left(\frac{u\pm1}{2}\right)+\frac{u\pm1}{2}=u.$$
Secondly,	 $H-A$ is  symmetric as $$(H-A)^{(-1)}=(H-A^{(-1)})=H-A.$$
Now,
\begin{align*}	
	S'(S')^{(-1)}=&\\ =&\,\,(H-A+\omega B)(H-A+\omega B)^{(-1)}\\=&\,\,(H-A)(H-A)^{(-1)}+BB^{(-1)}+\omega \left(B(H-A)^{(-1)}+
	(H-A)B^{(-1)}\right).
	\end{align*}
	Then,
\begin{align*}
(H-A)(H-A^{(-1)})+BB^{(-1)} =&\,\, HH - HA^{(-1)} - AH + AA^{(-1)} + BB^{(-1)} \\
=& \,\,	|H|H - HA - AH + AA^{(-1)} + BB^{(-1)} \\
=& \,\,	uH - 2|A|H + AA^{(-1)} + BB^{(-1)} \\
=& \,\,	(u - (u\pm1))H + AA^{(-1)} + BB^{(-1)} \\
=& \,\,	\mp H + AA^{(-1)} + BB^{(-1)}\\
=&\,\,\mp H+ (u \pm 1)(1)+\lambda(H-1)\\
=&\,\begin{cases}(\frac{u+1}{2}-1)H+(u-\frac{u+1}{2}+1)(1)  \quad   \text{if $|A|=|B|=\frac{u+1}{2}$,}\\
	(\frac{u-3}{2}+1)H +(u-\frac{u-3}{2}-1)(1) \quad  \text{if $|A|=|B|=\frac{u-1}{2}$}\end{cases}\\
=&\,\, \frac{u-1}{2}H+\frac{u+1}{2}(1),	
\end{align*}
and
\begin{align*}
\omega \left(B(H-A)^{(-1)}+
	(H-A)B^{(-1)}\right) =&\, \,  \omega (B+B^{-1})(H-A)=\omega (H\pm1)(H-A)\\ 
	=&\, \,	\omega \left( \left( \frac{u\pm1}{2} \right)H \mp A \right).
\end{align*}
	By examining $S'S'^{(-1)}=(H-A+\omega B)(H-A+\omega B)^{(-1)}$ we see that for $t \neq 1$ 
 the autocorrelation function of $\s'$ has the following form
\begin{equation}\label{eqn:perfectbalanced}
C_{\s'}(t)=
\frac{u\pm 1}{2}. 
\end{equation}
Thus, by equations~(\ref{perfect01}),~(\ref{balanced01}), 
and~(\ref{eqn:perfectbalanced}) the $G$-array $\s'$ is  perfect.
The case for $$S'=A+(H-B)\omega$$ is proven similarly.	
In this case, the skew-symmetry of $H-B$ follows from
\begin{align*}
(H-B)+(H-B)^{(-1)}
=&\,\, H-B+H-B^{(-1)} \\
=&\,\, 2H-(B+B^{(-1)}) \\
=&\,\, 2H-(H\pm1) \\
=&\,\, H\mp1. 
\end{align*}
\end{proof}
\subsection{The Ding-Helleseth-Martinsen $\{0,1\}$ $\mathbb{Z}_{2}\times \mathbb{Z}_p^{m}$-array based Yamada-Pott $\{0,1\}$ $\mathbb{Z}_p^{m}$-array pairs}\label{open}
 A Yamada-Pott $\{0,1\}$ $\mathbb{Z}_p^{m}$-array pair can be obtained from the array
pair located by Ding-Helleseth-Martinsen $\{0,1\}$  $\mathbb{Z}_2 \times \mathbb{Z}_p^{m}$-array in~\cite{DingHelleseth}
  for two cases, where  $p^m \equiv 5 \; (\text{mod} \;8)$, $p^m=s^2+4t^2$, 
   $s\equiv1$ (mod $4$) and $p$ is a prime. The two cases are  $s=1$ and $t^2=1$. When 
$t^2=1$ we get a Yamada-Pott $\{0,1\}$  $\mathbb{Z}_p^{m}$-array   pair, while in the 
$s=1$ case or  for any $p^m \equiv 5 \; (\text{mod} \;8)$, we get a Szekeres $\{0,1\}$ $\mathbb{Z}_p^{m}$-array  pair. First, we present the case of the Ding-Helleseth-Martinsen family of $s=1$ locating a Szekeres $\{0,1\}$ $\mathbb{Z}_p^{m}$-array pair for all $p^m \equiv 5 \; (\text{mod} \;8)$.

\begin{theorem} \label{ThmDHM_s1_SzekeresPair}
 For each prime power $q=p^m\equiv5 \, \, (\hspace{-.25cm}\mod 8)$ such that $q=s^2+4t^2=1+4t^2$, the 
 Ding-Helleseth-Martinsen $\{0,1\}$  $\mathbb{Z}_{2}\times \mathbb{Z}_p^{m}$-array locates the Szekeres $\{0,1\}$ $\mathbb{Z}_p^{m}
 $-array pair $(\aaa,\bb)$, where the  sets of $1$ indices  
 of $(\aaa, \bb)$ 
 are $$(A,B)=(C^{(4,q,\alpha)}_0 \cup C^{(4,q,\alpha)}_1,C^{(4,q,\alpha)}_0 \cup C^{(4,q,\alpha)}_3).$$
\end{theorem}
\begin{proof}
The fact that the Ding-Helleseth-Martinsen $\{0,1\}$  $\Theta(\mathbb{Z}_{2}\times \mathbb{Z}_p^{m})$-array locates the Szekeres $\{0,1\}$ $\Theta(\mathbb{Z}_p^{m})$-array pair $(\Theta((a_g)),\Theta((b_g)))$, whose sets of $1$ indices are 
$$(\Theta(C^{(4,q,\alpha)}_0 \cup C^{(4,q,\alpha)}_1),
 \Theta(C^{(4,q,\alpha)}_0 \cup C^{(4,q,\alpha)}_3)),$$ follows from the definition of the 
 Ding-Helleseth-Martinsen $\{0,1\}$  $\Theta(\mathbb{Z}_{2}\times \mathbb{Z}_p^{m})$-array for $s=1$. The result now follows from Lemma~\ref{iso}.
\end{proof}

  Next, we show that exactly one of the equivalence classes $1$ and  $2$ Ding-Helleseth-Martinsen family with $t^2=1$ 
  locates a $\{0,1\}$   Yamada-Pott 
 $\mathbb{Z}_p^{m}$-array  pair. 

Let $q=p^m$ for some prime $p$ and $n, D \in \mathbb{Z}$. Then  a representation $nq=x^2+Dy^2$  for some $x,y \in \mathbb{Z}$ is called a {\em proper} if $\gcd(q,x)=1$~\cite[p. 35]{StorerBook}. When $p\equiv 1$ (mod $4$) there are many representations of $q$ in the form $q=s^2+4t^2$ 
for some $s,t \in \mathbb{Z}$. However there is precisely one proper representation~\cite[p. 47]{StorerBook}.

Let $q=p^m=4\ell+1$ for some prime $p$ and odd positive integer $\ell$, or equivalently, let  $p$ be a a prime with $p\equiv 5$ (mod $8$) and $m\in 2\mathbb{Z}_{\geq0}+1$. Then  the unique proper representation of $q$ has the form $q=s^2+4t^2$ with $s \equiv 1$ (mod $4$)
and $t\in \mathbb{Z}$, where the sign of $t$ is 
undetermined~\cite[p. 51]{StorerBook}.  
Let $\alpha$ be a generator of $\FF_q^*$. 
Then, by  Lemma 19 in~\cite[p. 48]{StorerBook}  
\begin{equation}\label{talpha}
t(\alpha)=\frac{16\times(0,3)_{q,\alpha}^4-q-1-2s}{8},
\end{equation} where $t^2=(t(\alpha))^2$.

The integers
$(i,j)_{q, \alpha}^d=|(C_i^{(d, q, \alpha)}+1) \cap C_j^{(d, q, \alpha)}|$
are called the {\em cyclotomic numbers of order $d$} with respect to $\FF_q$
 and $\alpha$ such that $\FF_q^*=\langle \alpha\rangle$.
The following  lemma is needed to establish our results.
\begin{lemma} \label{lemDHM_t1_SzekeresPair}
Let $p$ be a prime, $p\equiv 5$ (mod $8$), $q=p^m$, and
 $m\in 2\mathbb{Z}_{\geq0}+1$. Let $q=s^2+4t^2$ be the unique proper representation of $q$. Let  
 \begin{eqnarray*}
(A_1,B_1)&=&(C^{(4,q,\alpha)}_0 \cup C^{(4,q, \alpha)}_1,C^{(4,q, \alpha)}_1 \cup C^{(4,q, \alpha)}_3),\\ (A_2,B_2)&=&(C^{(4,q, \alpha)}_0 \cup C^{(4,q, \alpha)}_2,  C^{(4,q, \alpha)}_2\cup C^{(4,q, \alpha)}_3),
\end{eqnarray*}
 where $\FF_q^*=\langle \alpha\rangle$ and $t(\alpha)$ is as in equation~(\ref{talpha}). Then  
   {\footnotesize
     \begin{equation} \label{Equation4Ding}
	|A_1\cap(A_1+x)|+|B_1\cap(B_1+x)|= \begin{cases}
	A+4E+2B+D=	\frac{q-t(\alpha)-2}{2}\quad     $ if $x^{-1} \in
	 C_{0}^{(4,q,\alpha)}\cup C_{2}^{(4,q,\alpha)},\\
	4A+2E+C+D=	\frac{q+t(\alpha)-4}{2} \quad $ if $x^{-1} \in 
	 C_{1}^{(4,q,\alpha)}\cup C_{3}^{(4,q,\alpha)},
	\end{cases}
\end{equation}}
and 
  {\footnotesize
     \begin{equation} \label{Equation5Ding}
	|A_2\cap(A_2+x)|+|B_2\cap(B_2+x)|= \begin{cases}
	4A+2E+B+C=	\frac{q-t(\alpha)-4}{2}\quad     $ if $x^{-1} 
	\in C_{0}^{(4,q,\alpha)}\cup C_{2}^{(4,q,\alpha)},\\
	4E+2D+A+B=	\frac{q+t(\alpha)-2}{2} \quad $ if $x^{-1} 
	\in  C_{1}^{(4,q,\alpha)}\cup C_{3}^{(4,q,\alpha)},
	\end{cases}
\end{equation}}
where
 {\footnotesize
    \begin{align*}
A=&\,\, \frac{q-7+2s}{16}, \\
B=&\,\,  \frac{q+1+2s-8t(\alpha)}{16}, \\
C=&\,\,  \frac{q+1-6s}{16}, \\
D=&\,\, \frac{q+1+2s+8t(\alpha)}{16},\\
E=& \,\, \frac{q-3-2s}{16}.
\end{align*}}
\end{lemma}
\begin{proof}
This result is proven in the proof of Theorem 3.1	in~\cite{Ding2008}.
(There are two typos in 
	 equation~(5) in~\cite{Ding2008}; ``$\frac{q-2-t}{2}$"
	 and ``$\frac{q-4+t}{2}$" should be ``$\frac{q-4-t}{2}$"
	 and ``$\frac{q-2+t}{2}$" respectively. Equation~(5) in~\cite{Ding2008} is  equation~(\ref{Equation5Ding}) here.)
\end{proof}
\begin{theorem} \label{ThmDHM_t1_SzekeresPair}
For $i=1,2$, let $(\aaa_i,\bb_i)$ be $\{0,1\}$ $\mathbb{Z}_p^{m}$-pair whose sets of $1$ indices are $(A_i,B_i)$ in Lemma~\ref{lemDHM_t1_SzekeresPair}.
Then $(\aaa_i,\bb_i)$ is a Yamada-Pott $\{0,1\}$ $\mathbb{Z}_p^{m}$-pair if and only if $t(\alpha)=(-1)^{i+1}$, where $t(\alpha)$ is as in equation~(\ref{talpha}). 
 Hence, exactly one of	the $(\aaa_i,\bb_i)$ is 
 a Yamada-Pott $\mathbb{Z}_p^{m}$-array pair. 
\end{theorem}
\begin{proof}
First, $(A_1,B_1)$ is  a $2$-SDS$(q; (q-1)/2, (q-1)/2, (q-3)/2)$ if and only if 
$$\frac{q-t(\alpha)-2}{2}=\frac{q+t(\alpha)-4}{2}=\frac{q-3}{2}\iff t(\alpha)=1,$$ and
$(A_2,B_2)$ is a $2$-SDS$(q; (q-1)/2, (q-1)/2, (q-3)/2)$ if and only if 
$$\frac{q-t(\alpha)-4}{2}=\frac{q+t(\alpha)-2}{2}=\frac{q-3}{2} \iff t(\alpha)=-1.$$
Hence,	 the choice of field generator $\alpha$ determines which pair is the supplementary difference set as $(0,3)_{q,\alpha}^4$ is a function 
of $\alpha$. To prove the symmetry of $B_1$ and the skew-symmetry of $A_1$ first observe that $q=s^2+4$, with $s\equiv1$ (mod $4$) implies $q=8j+5$ for some $j \in \mathbb{Z}^{\geq 0}$. Since $$
-1=\alpha^\frac{q-1}{2}=\alpha^{4j+2},$$ we have
$$-C^{(4,q, \alpha)}_1 = \alpha^{4j+2}\alpha C^{(4,q, \alpha)}_0=\alpha^3C^{(4,q, \alpha)}_0=
C^{(4,q, \alpha)}_3,$$	  and
	   $$-C^{(4,q, \alpha)}_0=
	   \alpha^{4j+2}C^{(4,q, \alpha)}_0=C^{(4,q, \alpha)}_2.$$ Then
 $$B_1^{(-1)} = 
	  (C^{(4,q, \alpha)}_1 + C^{(4,q, \alpha)}_3)^{(-1)} =-C^{(4,q, \alpha)}_1 - C^{(4,q, \alpha)}_3 =
	   C^{(4,q, \alpha)}_3+C^{(4,q, \alpha)}_1=B_1,$$ and $B_1$ is symmetric. Moreover, 
	   \begin{equation}\label{eqn:A_1}
	   A_1^{(-1)}=(C^{(4,q, \alpha)}_0+C^{(4,q, \alpha)}_1)^{(-1)}=-C^{(4,q, \alpha)}_0-C^{(4,q, \alpha)}_1=
	   C^{(4,q, \alpha)}_2+C^{(4,q, \alpha)}_3.
	   \end{equation}  Now, equation~(\ref{eqn:A_1}) implies $\{A_1\} \cap \{A_1^{(-1)}\}=\emptyset,$
	 and    $A_1+A_1^{(-1)}=\mathbb{Z}_p^m-0$.   Hence, $A_1$ is skew-symmetric.
The symmetry of $A_2$ and the skew-symmetry of $B_2$ are proven similarly. The result now follows from Theorem~\ref{thm:Legendre}.
\end{proof}
Let $C_{i,j,l}$ in Section~\ref{sec:families} be the sets of $1$ indices of pairs of $\mathbb{Z}_p^m$-arrays. It is easy to check that the equivalence classes of pairs of $\mathbb{Z}_p^m$-arrays whose sets of $1$ indices are $C_{0,1,3} , C_{0,2,3}$, and $C_{1,0,3}$ constitute all equivalence classes of all possible  pairs of $\mathbb{Z}_p^m$-arrays
whose sets of $1$ indices have the form  $C_{i,j,l}$.
Hence, Theorems~\ref{ThmDHM_s1_SzekeresPair} and~\ref{ThmDHM_t1_SzekeresPair} cover all equivalence classes of all possible such  $\mathbb{Z}_p^m$-arrays.

The following corollary provides two equivalent conditions to the equivalence class $i$ 
Ding-Helleseth-Martinsen family of $\Theta(\mathbb{Z}_{2}\times \mathbb{Z}_p^{m})$-array $\s_i$ being perfect.
\begin{corollary}\label{corollary}
Let $(\aaa_i,\bb_i)$ and $t(\alpha)$ be as in Theorem~\ref{ThmDHM_t1_SzekeresPair}. Let $\Theta(\mathbb{Z}_2)=\langle\omega\rangle$, and $S_i=\Theta(A_i)+\Theta(B_i)\omega$ be the set of $1$ indices of
the equivalence class $i$ 
Ding-Helleseth-Martinsen family of 
$\{0,1\}$ $\Theta(\mathbb{Z}_{2}\times \mathbb{Z}_p^{m})$-array $\s_i$. 
Then the following are equivalent:
\begin{itemize}
\item[(i)]
  The $\{0,1\}$ $\Theta(\mathbb{Z}_{2}\times \mathbb{Z}_p^{m})$-array $\s_i$  is perfect.
\item[(ii)]
 $(\aaa_i, \bb_i)$ is a Yamada-Pott $\{0,1\}$ $\mathbb{Z}_p^{m}$-array pair.
\item[(iii)]
$t(\alpha)=(-1)^{i+1}$.
\end{itemize}
\end{corollary}
\begin{proof}
The equivalence of $(ii)$ and $(iii)$ follows from Theorem~\ref{ThmDHM_t1_SzekeresPair}.
$(ii) \implies (i)$ follows from Theorem~\ref{ThmYPPairToADS}.
To prove $(i) \implies (ii)$, we already proved in the proof of Theorem 
~\ref{ThmDHM_t1_SzekeresPair}  that $\aaa_1$ and $\bb_2$ are skew-symmetric and 
$\bb_1$ and $\aaa_2$ are symmetric. 
So, it suffices to show that $(\Theta(\aaa_i),\Theta(\bb_i))$ is a Legendre pair. By the definition in equation~(\ref{perfect01}), $\s_i$ is perfect implies 
\begin{equation}\label{perfectcs}
C_{\s_i}(t)=\frac{q-1}{2} \quad \text{or} \quad \frac{q-3}{2}\quad \text{if}\quad t\neq 1.
\end{equation}
Now, 
\begin{align*}
\sum_{t \in \Theta(\mathbb{Z}_{2}\times \mathbb{Z}_p^{m})}C_{\s_i}(t)t=S_iS_i^{(-1)} =&\,\, \left( \Theta(A_i)+\omega\Theta(B_i) \right) \left( \Theta(A_i)+\omega\Theta(B_i) \right)^{(-1)} \\
=&\,\,\left( \Theta(A_i)+\omega\Theta(B_i) \right) \left( \Theta(A_i)^{(-1)}+\omega\Theta(B_i)^{(-1)} \right). 
\end{align*}
Then, 
{\small
\begin{equation}\label{eqn:SiSi}
S_iS_i^{-1} = \Theta(A_i)\Theta(A_i)^{(-1)} + \Theta(B_i)\Theta(B_i)^{(-1)} +  \omega\left(\Theta(A_i)\Theta(B_i)^{(-1)} + \Theta(A_i)^{(-1)}\Theta(B_i)  \right). 
\end{equation}}The isomorphism $\Theta:\mathbb{Z}_2\times\mathbb{Z}_p^m\rightarrow 
\Theta(\mathbb{Z}_2\times\mathbb{Z}_p^m)$ extends linearly to an isomorphism of $\mathbb{Z}[\mathbb{Z}_2\times\mathbb{Z}_p^m]$ and
$\mathbb{Z}[\Theta(\mathbb{Z}_2\times\mathbb{Z}_p^m)]$.
If $(\aaa_i,\bb_i)$ is not a Legendre pair then by equations (\ref{Equation4Ding}) and (\ref{Equation5Ding})   $$A_i(A_i)^{(-1)} + B_i(B_i)^{(-1)}$$ has terms whose coefficients are equal to $(q-5)/2$. Then   equation~(\ref{eqn:SiSi}) implies
$$\Theta(A_i)\Theta(A_i)^{(-1)} + \Theta(B_i)\Theta(B_i)^{(-1)}$$ has terms whose coefficients are  $(q-5)/2$,
and this contradicts equation~(\ref{perfectcs}). 
\end{proof}

By establishing $(i) \iff (iii)$ in Corollary~\ref{corollary} we also solved the second  of the proposed two open problems at the end of Section 3 in~\cite{Ding2008}. As far as we know this problem has been open until now.

The second part of Theorem 3.1 of~\cite{ArasuDicyclic}  states  that the 
  $\mathbb{Z}_p^{m}$-array pair $(\aaa,\bb)$ with sets of $1$ indices $(C^{(4,q,\alpha)}_0 \cup C^{(4,q,\alpha)}_1, C^{(4,q,\alpha)}_0 \cup  C^{(4,q,\alpha)}_2 )$ satisfies the Legendre $\mathbb{Z}_p^{m}$-array pair condition.
  This is not always true. On page 130 of~\cite{PottBook}, Pott incorrectly credits~\cite{ArasuDicyclic} for this theorem (as it works if and only if $t(\alpha)=-1$). Nevertheless, this does not impact the main theme of~\cite{ArasuDicyclic} on dicyclic designs.
  The following corollary corrects the second part of Theorem 3.1 of~\cite{ArasuDicyclic}.
  \begin{corollary}
Let $t(\alpha)$  and $q=p^m=s^2+4t(\alpha)^2$ be as in equation~(\ref{talpha}). Then, the  $\mathbb{Z}_p^{m}$-array pair $(\aaa,\bb)$ with sets of $1$ indices $D_1=C^{(4,q,\alpha)}_0 \cup C^{(4,q,\alpha)}_1$, and $D_2= C^{(4,q,\alpha)}_0 \cup  C^{(4,q,\alpha)}_2 $ satisfies the Legendre $\mathbb{Z}_p^{m}$-array pair condition if and only if $t(\alpha)=-1$. Moreover, $\aaa$ is  skew-symmetric and $\bb$ is symmetric.
  \end{corollary}
  \begin{proof}
  Let $A_2$, $B_2$ and $\aaa_2$, $\bb_2$ be as in 
  Theorem~\ref{ThmDHM_t1_SzekeresPair}. Then $\aaa_2$ 
  is symmetric and $\bb_2$ is skew-symmetric.
   Observe that $(D_1,D_2)=(\alpha^2B_2,\alpha^2A_2)$.
  Hence, $\mathbb{Z}_p^m$-array pair $(\bb_2,\aaa_2)$ is equivalent to $(\aaa,\bb)$. Thus, $(\aaa,\bb)$ is a Legendre  $\mathbb{Z}_p^{m}$-array pair if and only if $t(\alpha)=-1$.
  By Lemma~\ref{lem:symskew}, $(D_1,D_2)=(\alpha^2B_2,\alpha^2A_2)$ implies $\aaa$ is  skew-symmetric and $\bb$ is symmetric.
   \end{proof}  

\subsection{The Sidelnikov-Lempel-Cohn-Eastman $\mathbb{Z}_{q-1}$-array based\\ Yamada-Pott $\{0,1\}$ $\mathbb{Z}_{(q-1)/2}$-array pairs}
An interesting fact about the 
Sidelnikov-Lempel-Cohn-Eastman $\{0,1\}$ $\mathbb{Z}_{q-1}$-array and the 
Yamada Yamada-Pott $\{0,1\}$ $\mathbb{Z}_{(q-1)/2}$-array pair is that each pair can be obtained from the other.
%
\begin{theorem}\label{thm:ABCD}
For $q \geq 7$ and $q \equiv 3\; (\text{mod} \;4)$ let  $(A_1, B_1)$ and $(A_2\cup B_2)$ be the pair of sets of $1$
 indices 
of the Yamada Yamada-Pott $\{0,1\}$ 
$\mathbb{Z}_{(q-1)/2}$-array pair and the set of $1$ indices of the Sidelnikov-Lempel-Cohn-Eastman $\{0,1\}$ $\mathbb{Z}_{q-1}$-array, where
\begin{equation*}
	\begin{array}{l}
		A_1 = \left\lbrace \log_\alpha x  \; (\hspace{-.43cm}\mod\; \frac{q-1}{2})\,\,|\,\, x \in (C^2_0+1) \cap C^2_0 \right\rbrace,  \\
		B_1 = \left\lbrace \log_\alpha x   \; (\hspace{-.43cm}\mod\; \frac{q-1}{2})\,\,|\,\, x \in (C^2_0-1) \cap C^2_0 \right\rbrace,  \\
		A_2 = \left\lbrace \log_\alpha x  \; (\hspace{-.43cm}\mod\; \frac{q-1}{2})\,\,|\,\, x \in (C^2_1-1) \cap C^2_0 \right\rbrace,  \\
		B_2 = \left\lbrace \log_\alpha x   \; (\hspace{-.43cm}\mod\; \frac{q-1}{2})\,\,|\,\, x \in (C^2_1-1) \cap C^2_1 \right\rbrace.  \\
	\end{array}
\end{equation*}
Then, $A_1=B_2$, and $B_1=\mathbb{Z}_{\frac{q-1}{2}}\setminus A_2$.
\end{theorem}
\begin{proof}
Observe that  $$\alpha^{\frac{q-1}{2}}[(C^2_0 + 1) \cap  C^2_0]  = (\alpha^{\frac{q-1}{2}} C^2_0 + \alpha^{\frac{q-1}{2}}) \cap \alpha^{\frac{q-1}{2}} C^2_0.$$ Then, $q \equiv 3 \;(\text{mod} \; 4)$ implies $\alpha^{(q-1)/2}=-1 \notin C^2_0$ giving $$\alpha^{\frac{q-1}{2}} [(C^2_0 + 1) \cap  C^2_0] = (C^2_1 - 1 ) \cap C^2_1.$$ After taking the discrete logarithm and reducing modulo $(q-1)/2$, we get $A_1=B_2$.
Since $\FF_q = \{0\} \cup C^2_0 \cup C^2_1$ is a partitioning of $\FF_q$ 
and $\phi(x)=x-1$ is a one-to-one function from $\FF_q$ to $\FF_q$ we 
get 
\begin{equation}\label{eqn:partition}
\FF_q = \{-1\} \cup (C^2_0-1) \cup (C^2_1-1)
\end{equation}
as another partitioning of  $\FF_q$. Now, by equation~(\ref{eqn:partition}) and the fact that $-1 \notin C^2_0$, we get  
$$C^2_0=C^2_0\cap [(C^2_0-1) \cup (C^2_1-1)]=(C^2_0\cap (C^2_0-1)) \cup
(C^2_0\cap (C^2_1-1))$$ as a partitioning of $C^2_0$.
Then,  we get the set equations 
\begin{eqnarray*}
& &2 \mathbb{Z}_{\frac{q-1}{2}}=\log_\alpha(C^2_0) = \log_\alpha[(C^2_0 \cap (C^2_0-1)) \cup (C^2_0 \cap (C^2_1-1))]=\\
& & \log_\alpha[(C^2_0 \cap (C^2_0-1))] \cup\log_\alpha [(C^2_0 \cap (C^2_1-1))]= 2[B_1 \cup A_2]
\end{eqnarray*}
 as $$(C^2_0 \cap (C^2_0-1)) \cap (C^2_0 \cap (C^2_1-1))=\emptyset.$$
Hence,  
 $$ \mathbb{Z}_{\frac{q-1}{2}}=\frac{1}{2}\log_{\alpha}(C^2_0)=B_1 \cup A_2.$$
  It is also clear that $A_2 \cap B_1=\emptyset$. Thus, $$B_1 = \mathbb{Z}_{\frac{q-1}{2}} \setminus A_2 \quad \text{and} \quad A_2 = \mathbb{Z}_{\frac{q-1}{2}} \setminus B_1.$$ 
\end{proof}
Next, we  locate an
almost balanced perfect $\{0,1\}$ $\mathbb{Z}_{q-1}$-array pair based on a family of $\{0,1\}$ $\mathbb{Z}_{(q-1)/2}$-array pairs. 
 In fact, this result is presented partially in~\cite{Ding2008}, as Theorem 4.1. By Lemma~\ref{isomorphism} it suffices to locate an  almost 
balanced perfect $\{0,1\}$ $\Theta(\mathbb{Z}_{q-1})$-array.

\begin{theorem}\label{thm:backup}
Let $q \geq 7$ and $q=p^m\equiv 3 \;(\text{mod}\; 4)$, $\alpha$  be a generator of $\FF_q^*$.   
Let 
\begin{equation*}
	\begin{array}{l}
		A = \left\lbrace \log_\alpha x  \; (\hspace{-.43cm}\mod\; \frac{q-1}{2})\,\,|\,\, x \in (C^2_0-1) \cap C^2_0 \right\rbrace,  \\
		B = \left\lbrace \log_\alpha x   \; (\hspace{-.43cm}\mod\; \frac{q-1}{2})\,\,|\,\, x \in (C^2_0-1) \cap C^2_1 \right\rbrace \\
	\end{array}
\end{equation*}
 be the pair of sets of $1$ indices
 of the  $\{0,1\}$ 
$\mathbb{Z}_{(q-1)/2}$-array $(\aaa,\bb)$ pair.
Then $(\aaa,\bb)$ is a Yamada-Pott $\{0,1\}$ $\mathbb{Z}_{(q-1)/2}$-array pair. Let $\langle \omega\rangle=\Theta(\mathbb{Z}_2)$.   Then $\Theta(A)+\Theta(B)\omega \in \mathbb{Z}
[\Theta(\mathbb{Z}_2\times\mathbb{Z}_{(q-1)/2})]$ 
 corresponds to an almost 
balanced perfect 
$\{0,1\}$ $\Theta(\mathbb{Z}_2\times\mathbb{Z}_{(q-1)/2})$-array.
\end{theorem}
\begin{proof}
Let $A'= \sum_{g \in A/2}g$ and $B'= \sum_{g \in (B-1)/2}g$, where $A',B'\in \mathbb{Z}[\mathbb{Z}_{(q-1)/2}]$ and $(\aaa',\bb')$ be the corresponding  $\mathbb{Z}_{(q-1)/2}$-array pair.
 In Theorem 4.1 of~\cite{Ding2008} it is shown that 
$$A'(A')^{(-1)}+B'(B')^{(-1)}=\frac{q-7}{4}(\mathbb{Z}_{\frac{q-1}{2}}-0)+\frac{q-3}{2}(0).$$
Now, since $(\aaa',\bb')$ and $(\aaa,\bb)$ are equivalent
$\mathbb{Z}_{(q-1)/2}$-array pairs, we get that $(\aaa,\bb)$ is a Legendre $\{0,1\}$ $\mathbb{Z}_{(q-1)/2}$-array pair. 
Next, we show that $\aaa$ is a symmetric $\mathbb{Z}_{(q-1)/2}$-array. 
For a set $A\subseteq \mathbb{Z}_{(q-1)/2}$ let $-A=\{x\,|\, -x\in A\}$. 
For any $x \in (C^2_0-1) \cap C^2_0$ we have 
\begin{equation}\label{eqn:2j}
x=\alpha^{2i}-1=\alpha^{2j},
\end{equation}
 for some $i,j \in \mathbb{Z}$.  Multiplying both sides of equation~(\ref{eqn:2j}) by $x^{-1}=\alpha^{-2j}$ yields 
$$\alpha^{2i-2j}-\alpha^{-2j}=1$$ or $$\alpha^{-2j}=\alpha^{2i-2j}-1.$$ Then, $x^{-1} \in (C^2_0-1) \cap C^2_0 = A$ implying $-2j \in A$.
Hence, if $2j\in A$ then $2j \in -A$. 
Since $|(-A)| = |A|$ we get $A=-A$.
Finally, we show that $\bb$, equivalently $B$ is skew-symmetric. By the definition of $B$, any $x \in (C^2_0-1) \cap C^2_1$ satisfies 
\begin{equation}\label{eqn:2jplus1}
x = \alpha^{2i}-1 = \alpha^{2j+1},
\end{equation}
 for some $i,j \in \mathbb{Z}$.
 By multiplying both sides of equation~(\ref{eqn:2jplus1}) with $x^{-1}= \alpha^{-(2j+1)}$ we see that $$\alpha^{2i-2j-1}-\alpha^{-(2j+1)} = 1.$$ By rearranging terms we get $$\alpha^{-(2j+1)} = \alpha^{2i-2j-1} - 1,$$ and so $x^{-1} \in (C^2_1-1) \cap C^2_1$. Hence, $x^{-1} \notin C^2_0-1$, and  $-2j-1 \notin B$. Thus, if $b=2j+1 \in B$, then $-b \notin B$ giving that 
 $B \cap  (-B) = \emptyset.$ Now, $q \equiv 3$ (mod $4$) implies $\alpha^{(q-1)/2}=-1\notin C_0^2$. Then, 
 $$\alpha^{\frac{q-1}{2}}[(C_0^2-1)\cap C_1^2]=(C_1^2+1)\cap C_0^2$$ 
 and consequently $|(C_0^2-1)\cap C_1^2|=|(C_1^2+1)\cap C_0^2|$. 
  Now, by equation~(\ref{eqn:partition}) and the fact that
   $-1 \in C^2_1$, we get  
$$C^2_1=C^2_1\cap [(C^2_0-1) \cup (C^2_1-1)]=\{-1\}\cup(C^2_1\cap (C^2_0-1)) \cup
(C^2_1\cap (C^2_1-1))$$ as a partitioning of $C^2_1$.
Then,   the set equations 
\begin{eqnarray*}
& &2 \mathbb{Z}_{\frac{q-1}{2}}+1=\log_\alpha(C^2_1) = \log_\alpha[-1]\cup \log_\alpha[(C^2_1 \cap (C^2_0-1)) \cup (C^2_1 \cap (C^2_1-1))]=\\
& &\log_\alpha[-1]\cup \log_\alpha[(C^2_1 \cap (C^2_0-1))] \cup\log_\alpha [(C^2_1 \cap (C^2_1-1))]
\end{eqnarray*}
gives a partitioning of $2\mathbb{Z}_{(q-1)/2}+1$
 as $$(C^2_1 \cap (C^2_0-1)) \cap (C^2_1 \cap (C^2_1-1))=\emptyset,$$
 $-1 \notin (C^2_1 \cap (C^2_0-1))$ and $-1 \notin  (C^2_1 \cap (C^2_1-1))$.
 Since $\gcd ((q-1)/2,2)=1$, $\phi(x)=2x+1$ is an automorphism of 
 $\mathbb{Z}_{(q-1)/2}$.
 Then 
 {\footnotesize
 $$\left(2\mathbb{Z}_{\frac{q-1}{2}}+1\right) (\hspace{-.3cm}\mod\; \frac{q-1}{2})= 
 \left(\log_\alpha[-1]\cup \log_\alpha[(C^2_1 \cap (C^2_0-1))] \cup\log_\alpha [(C^2_1 \cap (C^2_1-1))]\right)(\hspace{-.3cm}\mod\; \frac{q-1}{2})$$}
\hspace{-.2cm} is a partitioning of $\mathbb{Z}_{(q-1)/2}$. This implies that 
 $|B|=|(C_0^2-1)\cap C_1^2|$.
   By part b of Lemma 6 in~\cite[p. 30]{StorerBook},  
  $|(C_0^2-1)\cap C_1^2|=|(C_1^2+1)\cap C_0^2|=(q-3)/4$. 
  Hence, $|B|=(q-3)/4$.
 We also have $|B|=|(-B)|$ and $B\cap(-B)=\emptyset$, 
 so $B\cup(-B)=\mathbb{Z}_{(q-1)/2}\backslash 0.$  
 Now, the result follows from Lemma~\ref{lem:symskewG}.
\end{proof}
\indent While it is believed that a Legendre $\{0,1\}$ $\mathbb{Z}_n$-array pair exists for all odd $n$, the existence of Yamada-Pott $\{0,1\}$ $\mathbb{Z}_n$-array pairs or $\{0,1\}$ $\mathbb{Z}_n$-array pairs $(\aaa,\bb)$ such that both $\aaa$ and $\bb$ are symmetric or skew-symmetric has not received as much attention. 
Table~\ref{tab:Table_YP_Pair_Cases} shows the existence and 
non-existence of $\{0,1\}$ Yamada-Pott 
$\mathbb{Z}_n$-array pairs. The comment column describes either how the pair is generated 
or how we have shown  nonexistence. ``Computer search'' means the existence or non-existence of a Yamada-Pott $\{0,1\}$ $\mathbb{Z}_n$-array was proven by an exhaustive computer search.  Under the ``Exist?'' column a ``Y'' or ``N'' means yes or no.
Our computer search was based on going through all possible pairs of $\{0,1\}$ sequences, $\aaa,\bb$ such that $$\sum_{i=1}^n a_i=
\sum_{i=1}^n b_i=\frac{n+1}{2}$$ and screening out the pairs that formed a Legendre pair. At the end of the search, for each found $\{0,1\}$ Legendre $\mathbb{Z}_n$-array pair $(\aaa,\bb)$, we checked for the  symmetry and skew symmetry of $\aaa$ and $\bb$ respectively.
\begin{table}[ht!]
	\begin{center}
		\caption{The existence of Yamada-Pott $\{0,1\}$ $\mathbb{Z}_n$-array pairs
		}\label{tab:Table_YP_Pair_Cases} 
		\begin{tabular}{|c|c|l}
		\hline 
		n & Exist? & Comment \\ 
		\hline 
		3 & Y & Theorem~\ref{thm:backup} with $q=2(3)+1=7$ \\ 
		\hline 
		5 & Y & Theorem~\ref{thm:backup} with $q=2(5)+1=11$  \\ 
		\hline 
		7 & N & Computer search \\ 
		\hline 
		9 & Y & Theorem~\ref{thm:backup} with $q=2(9)+1=19$ \\
		\hline 
		11 & Y & Theorem~\ref{thm:backup} with $q=2(11)+1=23$ \\
		\hline 
		13 & Y & Theorem~\ref{thm:backup} with $q=2(13)+1=27$ \\
		\hline 
		15 & Y & Theorem~\ref{thm:backup} with $q=2(15)+1=31$ \\
		\hline 
		17 & N & Computer search \\
		\hline 
		19 & N & Computer search  \\
		\hline 
		21 & Y & Theorem~\ref{thm:backup} with $q=2(21)+1=43$ \\
		\hline 
		23 & Y & Theorem~\ref{thm:backup} with $q=2(23)+1=47$ \\	
		\hline 
		25 & N & Computer search \\	
		\hline 
		27 & N & Computer search \\	
		\hline 
		29 & Y & Theorem~\ref{thm:backup} with $q=2(29)+1=59$ \\
		\hline 
		31 & N & Computer search \\
		\hline 
		\end{tabular} 
	\end{center}
\end{table}

Table~\ref{tab:atleastexist} shows the existence of a 
Legendre $\{0,1\}$ $\mathbb{Z}_n$-array pair for all possible combinations of $\aaa$ and $\bb$ being symmetric, skew-symmetric and neither symmetric nor skew-symmetric. 
The number at the top of each column is $n$.  The first two columns describe the attributes of $\aaa$ and $\bb$ respectively. In the first two columns ``N" means neither symmetric nor skew-symmetric, ``Sk" means skew-symmetric and  ``S" means symmetric.
   For each cell that is in a column with an integer at the top, ``E'' and  ``NE'' mean exists and does not exist respectively. 

\begin{table}[h]
	\begin{center}
	\caption{The existence of a 
Legendre $\{0,1\}$ $\mathbb{Z}_n$-array pair for all possible combinations of $\aaa$ and $\bb$ being symmetric, skew-symmetric and neither symmetric nor skew-symmetric} \label{tab:atleastexist}
		\begin{tabular}{cc|ccccccccc} 
	\multicolumn{2}{c}{Type}& \multicolumn{9}{c}{n}  \\ 
		\hline                 
       A & B  &  5& 7& 9& 11& 13&15 &17 &19 &21 \\		
		\hline 
		N&N & E  & E& E & E&E &E &E & E&E \\ 
		\hline 
		N&S & E  & NE& E & E&E &E &E & E&E \\ 
		\hline 
		N&Sk & E  & E& E & E&E &E &NE & E&E \\ 
		\hline 
	 S&S & E  & NE& NE & NE&E &NE &E & NE&NE \\
		\hline 
		S&Sk & E  & NE& E & E&E &E &NE & NE&E \\
		\hline 
		Sk&Sk & E  & E& NE & E& E &NE &NE & E&NE \\
		\hline 
		\end{tabular} 	
	\end{center}
\end{table}

Exhaustive searches proved that no balanced, perfect  $\{0,1\}$ $\mathbb{Z}_{54}$-array exists, on two different supercomputers, with different programs~\cite{IliasEmail}. This is consistent with our computer searches as finding a Yamada-Pott $\{0,1\}$ $\mathbb{Z}_{27}$-array  pair would imply a perfect balanced $\{0,1\}$ $\mathbb{Z}_{54}$-array by Theorem~\ref{ThmYPPairToADS}.

We end this section with a couple of comments.
\begin{enumerate}
	\item In~\cite{PottBook}, on page $130$, it is claimed  that a Yamada-Pott $\{0,1\}$ $\mathbb{Z}_{37}$-array pair exists. This is false as it originated from the mistake in part 2 of Theorem 3.1 in~\cite{ArasuDicyclic}. 
	The case  $n=37$ with $q=2*37+1=75$ is not a prime power. 
	However, a symmetric Paley $\{0,1\}$ $\mathbb{Z}_{37}$-array pair and a skew-symmetric Szekeres $\{0,1\}$ $\mathbb{Z}_{37}$-array pair exist.
	 The first example  of a 
	 Yamada-Pott $\{0,1\}$ $\mathbb{Z}_{q}$-array pair that arises from Theorem \ref{ThmDHM_t1_SzekeresPair} is at 
	 $q=15^2+4=229$, where $n=2q+1=459=27*17$ is not a prime 
	 power.	 
\item Exhaustive searches proved that no Legendre $\{0,1\}$
$\mathbb{Z}_{7}$-array pair $(\aaa,\bb)$, where $\aaa$ is symmetric, and  no Legendre 
$\{0,1\}$ $\mathbb{Z}_{17}$-array pair $(\aaa,\bb)$, where $\aaa$ is skew-symmetric
exists. Exhaustive searches  found a Legendre $\{0,1\}$ $\mathbb{Z}_{n}$-array pair $(\aaa,\bb)$, where at least one of $\aaa$ and $\bb$ is symmetric or skew-symmetric for each $n\leq 21$, see Table~\ref{tab:atleastexist}.
\end{enumerate}

\subsection{An inequivalent Legendre $\{-1,1\}$ $\mathbb{Z}_{57}$-array pair}
By using  a heuristic computer search the only known example of a Legendre $\{-1,1\}$ $\mathbb{Z}_{57}$-array pair was found in~\cite{KotsireasLP57}.   This  resulted in the construction of a $116 \times 116$ Hadamard matrix
 via Theorem~\ref{thm:bordered}. 
The Legendre $\{-1,1\}$ $\mathbb{Z}_{57}$-array pair  found is given by \\ {\scriptsize $$\aaa_1=(- + + - + - + - + + - + + - - - - + - + + + + + - + + + - - + + + - + - - - - - + - + - - + - + + + - - + - + - -)^{\top},$$}  {\scriptsize $$\bb_1=(- - + + - - + + - + + + + - + + - + + + + + - - - - - + - + + + - + + - - + - + - - - + + + - - - - + + - - - - +)^{\top},$$}where $-$, $+$ are used for $-1$, $1$, and commas are deleted to save space. 
 \\ This pair can be shown to satisfy the condition given by Definition \ref{DefLegendrePair}.  The distributions of the autocorrelations of $A$ and $B$ are $$(-11)^2(-7)^{12}(-3)^{10}(1)^{20}(5)^{12},$$ and $$(-7)^{12}(-3)^{20}(1)^{10}(5)^{12}(9)^2.$$ 
 
 By using cyclotomy we found a Legendre 
$\{-1,1\}$ $\mathbb{Z}_{3}\times\mathbb{Z}_{19}$-array pair $(\x, \y)$ that can be used to construct a  Legendre $\{-1,1\}$ $\mathbb{Z}_{57}$-array pair that is not equivalent to  the previously known $\{-1,1\}$ Legendre $\mathbb{Z}_{57}$-array pair. This construction is displayed in the next example. 
 \begin{example}\label{ex57}
 Construct $C_i^{(6,19)}$ for $i=0,1,\ldots,5$ for $\alpha=2$.
For this example, we explicitly construct these cosets for $d=6$, $q=19$ and $\alpha=2$. 
The elements are given by $C_0^{(6,19,2)}=\{1,7,11\}$ with the remaining cosets being generated  by multiplying  $C_0^{(6,19,2)}$ by $\alpha=2$ and reducing modulo $19$.
For brevity, we use $C_i^{6}$ for $C_i^{(6,19,2)}$.  Let 
$$X=\left\{\{0\}\times\{0, C^{6}_0, C^{6}_1, C^{6}_2\}\right\} \cup \left\{\{1\} \times \{C^{6}_0, C^{6}_2, C^{6}_3, C^{6}_4\}\right\} \cup \left\{\{2\}\times\{C^{6}_3, C^{6}_4\}\right\},$$
and
$$Y=\left\{\{0\}\times\{0, C^{6}_0, C^{6}_4, C^{6}_5\}\right\} \cup \left\{\{1\} \times \{C^{6}_0, C^{6}_3, C^{6}_5\}\right\} \cup \left\{\{2\}\times\{C^{6}_0, C^{6}_1, C^{6}_3\}\right\}.$$ 
Then, the Legendre $\{-1,1\}$ $\mathbb{Z}_{3}\times \mathbb{Z}_{19}$-array pair $(\x,\y)$ obtained by letting 
\begin{equation*} \label{examplex}
	x_i= \begin{cases}
		\phantom{-} 1 \,\, \, \, \, \, \, \quad     $ if $i \in X,$ $\\
		-1 \,\,  \quad \quad \, $otherwise, $
	\end{cases}
\end{equation*}
and 
\begin{equation*} \label{exampley}
	y_i= \begin{cases}
		\phantom{-}1 \,\, \, \, \, \, \, \quad     $ if $i \in Y,$ $\\
		-1 \,\,  \quad \quad \, $otherwise $
	\end{cases}
\end{equation*}
 satisfies Definition~\ref{DefLegendrePair}. 
The  distribution of autocorrelations  for $\x$ and $\y$ are $$(-7)^{14}(-3)^{12}(1)^{18}(5)^{12},$$ $$(-7)^{12}(-3)^{18}(1)^{12}(5)^{14}.$$  
 The correlation energy for the 
 $\{-1,1\}$ $\mathbb{Z}_{3}	\times \mathbb{Z}_{19}$  
 array pair $(\x, \y)$ is $1,112$.
To construct a  Legendre $\{-1,1\}$ $\mathbb{Z}_{57}$ array pair
observe that $\mathbb{Z}_{57} \cong \mathbb{Z}_{3}	\times \mathbb{Z}_{19}$ via the map $\phi(i)  = (i\, \, (\text{mod } 3),\, (i\, \, (\text{mod } 19))$. Then, the $\{-1,1\}$  $\mathbb{Z}_{57}$-array pair  $(\phi^{-1}((x_g)),\phi^{-1}((y_g)))$ is a  Legendre pair that has the same distribution of autocorrelations and the same correlation energy as those of $(\x,\y)$. The map $\phi^{-1}$ is constructed as follows. Let $\phi(i)=(i\, \, (\text{mod } 3),\, (i\, \, (\text{mod } 19))=(a,b)$. Then there exists $k_1,k_2 \in \mathbb{Z}^{\geq 0}$ such that $a+3k_1=b+19k_2=i$, $0\leq k_1\leq 19$, and $0\leq k_2 \leq 3$. Then, $3k_1+(a-b)=19k_2$ imply $k_2=(a-b)19^{-1}\, \, (\text{mod } 3)$ and $k_1=-(a-b)3^{-1}\, \, (\text{mod } 19)$. Now, $k_1,k_2$ are uniquely determined by the inequalities $0\leq k_1\leq 19$, and $0\leq k_2 \leq 3$. Hence, $\phi^{-1}(a,b)=a+3k_1=b+19k_2=i$. This gives us 
{\scriptsize 
$$\aaa_2=( +  + -  +  +  +  +  +  +  + - - - - - -  +  + - - -  + - - - + - - + -  +  + -  + -  + -  + -  + - -  +  + -  +  + - -  + + - - - -  +  + )^{\top},$$}
 {\scriptsize $$\bb_2=(+  +  + - - - -  +  + -  +  + -  +  +  + - - - -  + - - -  + -  + - - -  +  + - -  + -  +  + -  + - + - - - + + -  +  +  +  + - -  + -  +)^{\top}.$$}The correlation energy for the Legendre 
$\{-1,1\}$ $\mathbb{Z}_{57}$-array pair in~\cite{KotsireasLP57} is $1,240.$ 
Thus $(\aaa_2,\bb_2)$ is not equivalent to $(\aaa_1,\bb_1)$.
 \end{example}
We propose developing theoretical  and computational methods for finding Legendre $\{-1,1\}$ $\mathbb{Z}_n$-array pairs by using cyclotomic cosets as in Example~\ref{ex57}  as a future research direction. 
 




\section*{Acknowledgements}
Dr. K.~T.~Arasu was supported by the
U.S. Air Force Research Lab Summer Faculty Fellowship Program
sponsored by the Air Force Office of Scientific Research.
The views expressed in this article are
those of the authors and do not reflect the official policy or
position of the United States Air Force, Department of Defense, or
the US Government.  The authors thank two referees for improving the clarity of the paper substantially.       

\bibliographystyle{spbasic}
\bibliography{Unification_Bib}
\end{document}